\documentclass[a4paper]{amsart}
\usepackage[foot]{amsaddr}

\usepackage[latin1]{inputenc}
\usepackage{amssymb,amsmath,amsthm}
\usepackage{amsfonts}
\usepackage{amsaddr}
\usepackage{enumitem}
\usepackage{booktabs}
\usepackage{ifthen}
\usepackage{tikz}
\usetikzlibrary{matrix,calc,decorations.markings}
\usepackage{url}
\usepackage{hyphenat}
\usepackage{hyperref}

\theoremstyle{plain}
\newtheorem{theorem}{Theorem}[section]
\newtheorem{proposition}[theorem]{Proposition}
\newtheorem{lemma}[theorem]{Lemma}

\theoremstyle{definition}
\newtheorem{definition}[theorem]{Definition}
\newtheorem{example}[theorem]{Example}

\newtheorem{remark}[theorem]{Remark}

\newcommand{\IN}{\ensuremath{\mathbb{N}}}

\newcommand{\nset}[1]{\ensuremath{[{#1}]}}
\newcommand{\Nout}[2][]{\ensuremath{N_\mathrm{o}^{#1}({#2})}}
\newcommand{\GA}[1]{\ensuremath{\mathbb{A}({#1})}}
\newcommand{\divides}{\ensuremath{\mathrel{\!|\!}}}
\newcommand{\card}[1]{\ensuremath{\lvert{#1}\rvert}}

\newcommand{\Htt}[1][t,t']{\ensuremath{\ifthenelse{\equal{#1}{}}{H}{H_{#1}}}}
\newcommand{\Mtt}[1][t,t']{\ensuremath{\ifthenelse{\equal{#1}{}}{M}{M_{#1}}}}
\newcommand{\Ltt}[1][t,t']{\ensuremath{\ifthenelse{\equal{#1}{}}{L}{L_{#1}}}}
\newcommand{\Ytt}[1][t,t']{\ensuremath{\ifthenelse{\equal{#1}{}}{Y}{Y_{#1}}}}
\newcommand{\Ztt}[1][t,t']{\ensuremath{\ifthenelse{\equal{#1}{}}{Z}{Z_{#1}}}}
\newcommand{\Deltatt}[1][t,t']{\ensuremath{\ifthenelse{\equal{#1}{}}{\Delta}{\Delta_{#1}}}}
\newcommand{\OMEGAtt}[1][t,t']{\ensuremath{\ifthenelse{\equal{#1}{}}{\Omega}{\Omega_{#1}}}}
\newcommand{\omegatt}[1][t,t']{\ensuremath{\ifthenelse{\equal{#1}{}}{\omega}{\omega_{#1}}}}
\newcommand{\Lambdatt}[1][t,t']{\ensuremath{\ifthenelse{\equal{#1}{}}{\Lambda}{\Lambda_{#1}}}}
\newcommand{\lambdatt}[1][t,t']{\ensuremath{\ifthenelse{\equal{#1}{}}{\lambda}{\lambda_{#1}}}}

\newcommand{\MG}[1][G]{\ensuremath{\ifthenelse{\equal{#1}{}}{M}{M_{#1}}}}
\newcommand{\PG}[1][G]{\ensuremath{\ifthenelse{\equal{#1}{}}{P}{P_{#1}}}}
\newcommand{\EG}[1][G]{\ensuremath{\ifthenelse{\equal{#1}{}}{E}{E_{#1}}}}
\newcommand{\OG}[1][G]{\ensuremath{\ifthenelse{\equal{#1}{}}{O}{O_{#1}}}}
\newcommand{\ZG}[1][G]{\ensuremath{\ifthenelse{\equal{#1}{}}{Z}{Z_{#1}}}}
\newcommand{\BG}[1][G]{\ensuremath{\ifthenelse{\equal{#1}{}}{B}{B_{#1}}}}
\newcommand{\omegaG}[1][G]{\ensuremath{\ifthenelse{\equal{#1}{}}{\omega}{\omega_{#1}}}}
\newcommand{\lambdaG}[1][G]{\ensuremath{\ifthenelse{\equal{#1}{}}{\lambda}{\lambda_{#1}}}}

\newcommand{\GRcycle}[1]{\ensuremath{\mathsf{C}_{#1}}}

\DeclareMathOperator{\lcm}{lcm}

\newcommand{\PROPtwoone}{2.1}
\newcommand{\PROPtwofive}{2.5}
\newcommand{\PROPtwosix}{2.6}
\newcommand{\LEMthreeone}{3.1}
\newcommand{\THMthreethree}{3.3}
\newcommand{\PROPfourone}{4.1}
\newcommand{\DEFfourtwo}{4.2}
\newcommand{\LEMfourfour}{4.4}
\newcommand{\LEMfoureight}{4.8}
\newcommand{\LEMfournine}{4.9}
\newcommand{\LEMfoureleven}{4.11}
\newcommand{\PROPfivefour}{5.4}
\newcommand{\REMfivefive}{5.5}
\newcommand{\LEMfivesix}{5.6}
\newcommand{\PROPfiveeight}{5.8}
\newcommand{\PROPfivenine}{5.9}
\newcommand{\FIGfive}{Figure~3 of Part~I}      % reference to the pre-review version
\begin{document}

\title[Associative spectra of graph algebras II]{Associative spectra of graph algebras II. \\ Satisfaction of bracketing identities, \\ spectrum dichotomy}

\author{Erkko Lehtonen}

\address[E. Lehtonen]%
   {Centro de Matem\'atica e Aplica\c{c}\~oes \\
    Faculdade de Ci\^encias e Tecnologia \\
    Universidade Nova de Lisboa \\
    Quinta da Torre \\
    2829-516 Caparica \\
    Portugal}

\author{Tam\'as Waldhauser}

\address[T. Waldhauser]%
   {University of Szeged \\
    Bolyai Institute \\
    Aradi v\'ertan\'uk tere 1 \\
    H-6720 Szeged \\
    Hungary}

\begin{abstract}
A necessary and sufficient condition is presented for a graph algebra to satisfy a bracketing identity.
The associative spectrum of an arbitrary graph algebra is shown to be either constant or exponentially growing.
\end{abstract}

\maketitle

\setcounter{section}{5}

%%%%%%%%%%%%%%%%%%%%%%%%%%%%%%%%%%%%%%%%%%%%%%%%%%

\section{Introduction to Part II}

This paper continues our study, initiated in \cite{PartI}, of associative spectra of graph algebras.
Introduced by Cs\'ak\'any and Waldhauser \cite{CsaWal-2000}, the associative spectrum of a binary operation or of the corresponding groupoid is a method of quantifying the degree of (non)\hyp{}associativity of the operation.
Graph algebras were introduced by Shallon~\cite{Shallon} as a way of encoding an arbitrary directed graph as an algebra with a binary operation.
We refer the reader to the first part of this study \cite{PartI} -- henceforth called ``Part~I'' -- for formal definitions, background, motivations, and further details that will not be repeated in this outline.
We continue the numbering of sections from Part~I, so that we can conveniently refer to theorems, definitions, etc.\ of Part~I simply by their numbers.

In Part~I, we determined the possible associative spectra of undirected graphs and classified undirected graphs by their spectra; there are only three distinct possibilities: constant $1$, powers of $2$, and Catalan numbers.
Furthermore, we characterized the antiassociative digraphs, and we determined the associative spectra of certain families of digraphs, such as paths, cycles, and graphs on two vertices.

In this paper, we turn our attention to graph algebras associated with arbitrary digraphs, which may be finite or infinite.
In Section~\ref{sec:satisfaction}, we provide a necessary and sufficient condition for a graph algebra to satisfy a nontrivial bracketing identity.
The condition is expressed in terms of several numerical structural parameters associated, on the one hand, with the digraph and, on the other hand, with a pair of bracketings.
We discuss in Section~\ref{sec:special} how some of the results of Part~I are obtained as special cases of this condition.

This result seems a first step towards a general description of the associative spectra of graph algebras associated with arbitrary digraphs.
Such a general result, however, eludes us.
We can nevertheless establish bounds for the possible associative spectra of graph algebras.
As we will see in Section~\ref{sec:dichotomy}, the associative spectrum of a graph algebra is either a constant sequence bounded above by $2$ or it grows exponentially,
the least possible growth rate of an exponential spectrum being $\alpha^n$, where $\alpha \approx 1.755$ is the following cubic algebraic integer:
\[
\alpha = \frac{1}{3}\sqrt[3]{\frac{25+3\sqrt{69}}{2}}+\frac{1}{3}\sqrt[3]{\frac{25-3\sqrt{69}}{2}}+\frac{2}{3}.
\]
This stands in stark contrast with associative spectra of arbitrary groupoids, where various subexponential spectra such as polynomials of arbitrary degrees are possible.

%%%%%%%%%%%%%%%%%%%%%%%%%%%%%%%%%%%%%%%%%%%%%%%%%%

\section{Satisfaction of bracketing identities by digraphs}
\label{sec:satisfaction}

We now turn to the general case of arbitrary directed graphs.
We are going to define several numerical parameters pertaining, on the one hand, to a pair of distinct bracketing terms $t, t' \in B_n$ and, on the other hand, to a digraph $G$.
For easy reference, the various parameters are collected in Table~\ref{table:parameters} with cross\hyp{}references to their definitions.
With the help of these parameters, we can provide necessary and sufficient conditions for the graph algebra of a digraph to satisfy a bracketing identity.
These conditions are put together in Theorem~\ref{thm:main-parameters}.

\begin{table}
\begin{center}
\begin{tabular}{cccc}
\toprule
parameter & Definition & parameter & Definition \\
\midrule
$\Htt$ & \DEFfourtwo   &   $\MG$ & \ref{def:MG} \\
$\Mtt$ & \DEFfourtwo   &   $\PG$ & \ref{def:PG} \\
$\Ltt$ & \DEFfourtwo   &   $\EG$ & \ref{def:EG} \\
$\Ytt$ & \ref{def:Y}   &   $\OG$ & \ref{def:OG} \\
$\Ztt$ & \ref{def:Z}   &   $\ZG$ & \ref{def:ZG} \\
$\omegatt$ & \ref{def:omega} & $\BG$ & \ref{def:BG} \\
$\lambdatt$ & \ref{def:lambda} & $\omegaG$ & \ref{def:omegaG} \\
       &               &   $\lambdaG$ & \ref{def:lambdaG} \\
\bottomrule
\end{tabular}
\end{center}

\medskip
\caption{Parameters of pairs of bracketings and graphs.}
\label{table:parameters}
\end{table}

Recall the parameters $\Htt$, $\Mtt$, and $\Ltt$ from Definition~\DEFfourtwo.
The following lemma extends Lemma~\LEMfourfour{}.

\begin{lemma}
\label{lem:s}
Let $t, t' \in B_n$, $t \neq t'$, and
let $G$ be a digraph such that $\GA{G}$ satisfies the identity $t \approx t'$.
Denote $H := \Htt$, $M := \Mtt$, $L := \Ltt$.
Let $r$ be the integer provided by Lemma~\LEMfourfour.
Then there exists an integer $s$ with $L + 1 \leq s \leq r$ and $s \equiv L \pmod{M}$ such that the following
holds:
if $v_0 \rightarrow v_1 \rightarrow \dots \rightarrow v_H$ and $v_L \rightarrow v'_{L+1} \rightarrow v'_{L+2} \rightarrow \dots \rightarrow v'_H$ are walks in $G$, then $v_s \rightarrow v'_{L+1}$ and $v'_s \rightarrow v_{L+1}$ are edges in $G$.
In particular, $v_{L+1}$ and $v'_{L+1}$ belong to the same nontrivial strongly connected component.
\end{lemma}

\begin{proof}
By the definition of $L$, there exists a vertex $x_d \in X_n$ such that either $d_T(x_d) = L + 1 < d_{T'}(x_d)$ or $d_{T'}(x_d) = L + 1 < d_T(x_d)$.
By changing the roles of $T$ and $T'$, if necessary, we may assume that $d_T(x_d) = L + 1 < d_{T'}(x_d)$.
Let $x_p$ be the parent of $x_d$ in $T$, and let $x_q$ be the parent of $x_d$ in $T'$.

Assume that $v_0 \rightarrow v_1 \rightarrow \dots \rightarrow v_H$ and $v_L \rightarrow v'_{L+1} \rightarrow v'_{L+2} \rightarrow \dots \rightarrow v'_H$ are walks in $G$.
By Lemma~\LEMfourfour, $v_r \rightarrow v_{L+1}$ and $v'_r \rightarrow v'_{L+1}$ are edges, so $v_{L+1} \rightarrow \dots \rightarrow v_r \rightarrow v_{L+1}$ and $v'_{L+1} \rightarrow \dots \rightarrow v'_r \rightarrow v'_{L+1}$ are closed walks in $G$.
Let $W$ be the walk that starts with $v_0 \rightarrow \dots \rightarrow v_L$ and continues by going around the closed walk $v_{L+1} \rightarrow \dots \rightarrow v_r \rightarrow v_{L+1}$ until it reaches length $h(T)$, and let $W'$ be the closed walk $v'_{L+1} \rightarrow \dots \rightarrow v'_r \rightarrow v'_{L+1}$.
Let $\varphi \colon X_n \to V(G)$ be the collapsing map of $(T,x_d)$ on $(W,W')$.
Since $\varphi$ is a homomorphism of $T$ into $G$, it is also a homomorphism of $T'$ into $G$ by Proposition~\PROPtwoone.
Since $(x_q, x_d) \in E(T')$, we have $(\varphi(x_q), \varphi(x_d)) \in E(G)$.
By definition, $\varphi(x_d) = v'_{L+1}$.
In order to determine $\varphi(x_q)$, note first that $q < d$ because $(x_q, x_d)$ is an edge in $T'$.
This implies that $x_q \notin T_{x_d}$ and thus $\varphi(x_q)$ lies in $W$, so $\varphi(x_q) = v_s$ for some $s \in \{0, 1, \dots, r\}$.
Since $d_T(x_q) \geq L+1$, $\varphi(x_q)$ lies on the closed walk $v_{L+1} \rightarrow \dots \rightarrow v_r \rightarrow v_{L+1}$.
Therefore $s$ is the unique element of the set $\{L+1, \dots, r\}$ such that $s \equiv d_T(x_q) \pmod{r-L}$; note that the value of $s$ does not depend on the walks $v_0 \rightarrow v_1 \rightarrow \dots \rightarrow v_H$ and $v_L \rightarrow v'_{L+1} \rightarrow v'_{L+2} \rightarrow \dots \rightarrow v'_H$ but only on $t$ and $t'$.
Since $r \equiv L \pmod{M}$, the number $r - L$ is divisible by $M$; therefore $s \equiv d_T(x_q) \equiv L \pmod{M}$.

Switching the roles of the closed walks $v_{L+1} \rightarrow \dots \rightarrow v_r \rightarrow v_{L+1}$ and $v'_{L+1} \rightarrow \dots \rightarrow v'_r \rightarrow v'_{L+1}$,
a similar argument shows that $(v'_s, v_{L+1}) \in E(G)$.
Now we have the closed walk $v_{L+1} \rightarrow \dots \rightarrow v_s \rightarrow v'_{L+1} \rightarrow \dots \rightarrow v'_s \rightarrow v_{L+1}$ in $G$.
This means, in particular, that $v_{L+1}$ and $v'_{L+1}$ belong to the same nontrivial strongly connected component.
\end{proof}

\begin{definition}
\label{def:MG}
For a digraph $G$, let $\MG$ be the least common multiple of the set of all numbers $m$ for which there exists a strongly connected component of $G$ that is an $m$\hyp{}whirl, with the convention that the least common multiple of the empty set is $1$.
If there is no finite upper bound on such numbers $m$, then define $\MG := \infty$.
\end{definition}

\begin{example}
Consider the graph $G$ shown in Figure~\ref{fig:G-param}.
Highlighted as shaded regions, the nontrivial strongly connected components are a $3$\hyp{}whirl and a $4$\hyp{}whirl.
Consequently, $\MG = \lcm(3,4) = 12$.
\end{example}

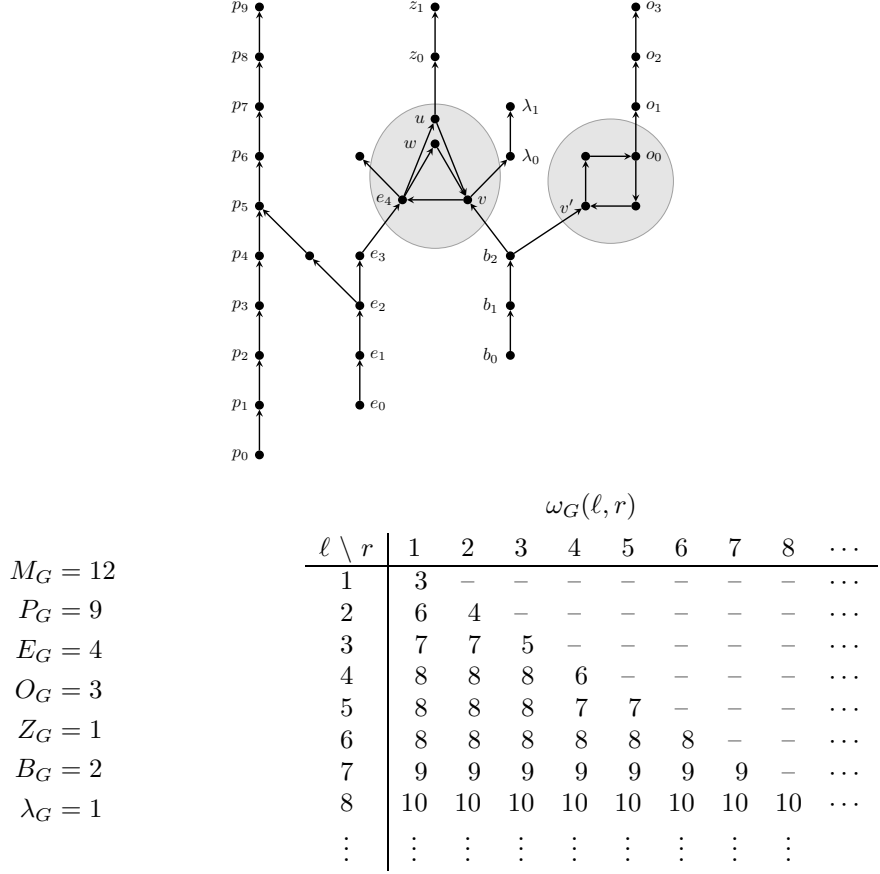
\begin{figure}
\begin{center}
\tikzset{every node/.style={circle,draw,inner sep=1.5,fill=black}, every path/.style={->,>=stealth,thick}}
\scalebox{0.66}{
\begin{tikzpicture}[scale=1]
\draw[thin,fill opacity=0.2,draw opacity=0.5,draw=gray!60!black,fill=gray] (3.5,5.6) circle [x radius = 1.3, y radius = 1.45];
\draw[thin,fill opacity=0.2,draw opacity=0.5,draw=gray!60!black,fill=gray] (7,5.5) circle [radius = 1.25];
\node[label=left:{$p_0$}] (a1) at (0,0) {};
\node[label=left:{$p_1$}] (a2) at (0,1) {};
\node[label=left:{$p_2$}] (a3) at (0,2) {};
\node[label=left:{$p_3$}] (a4) at (0,3) {};
\node[label=left:{$p_4$}] (a5) at (0,4) {};
\node[label=left:{$p_5$}] (a6) at (0,5) {};
\node[label=left:{$p_6$}] (a7) at (0,6) {};
\node[label=left:{$p_7$}] (a8) at (0,7) {};
\node[label=left:{$p_8$}] (a9) at (0,8) {};
\node[label=left:{$p_9$}] (a10) at (0,9) {};
\node (b5) at (1,4) {};
\node[label=right:{$e_0$}] (c2) at (2,1) {};
\node[label=right:{$e_1$}] (c3) at (2,2) {};
\node[label=right:{$e_2$}] (c4) at (2,3) {};
\node[label=right:{$e_3$}] (c5) at (2,4) {};
\node[label=left:{$e_4$}] (N11a) at ($(3.5,5.5)+(210:0.75)$) {};
\node[label=left:{$w$\,\,\,}] (N12a) at ($(3.5,5.5)+(90:0.75)$) {};
\node[label=left:{$u$}] (N12b) at ($(3.5,5.5)+(90:1.25)$) {};
\node[label=right:{$v$}] (N13a) at ($(3.5,5.5)+(330:0.75)$) {};
\node (c7) at (2,6) {};
\node[label=left:{$z_0$}] (d1) at (3.5,8) {};
\node[label=left:{$z_1$}] (d2) at (3.5,9) {};
\node[label=left:{$b_0$}] (e1) at (5,2) {};
\node[label=left:{$b_1$}] (e2) at (5,3) {};
\node[label=left:{$b_2$}] (e3) at (5,4) {};
\node[label=right:{$\lambda_0$}] (e4) at (5,6) {};
\node[label=right:{$\lambda_1$}] (e5) at (5,7) {};
\node[label=left:{$v'$}] (N21) at (6.5,5) {};
\node (N22) at (6.5,6) {};
\node (N24) at (7.5,5) {};
\node[label=right:{$o_0$}] (N23) at (7.5,6) {};
\node[label=right:{$o_1$}] (f1) at (7.5,7) {};
\node[label=right:{$o_2$}] (f2) at (7.5,8) {};
\node[label=right:{$o_3$}] (f3) at (7.5,9) {};
\path (a1) edge (a2);
\path (a2) edge (a3);
\path (a3) edge (a4);
\path (a4) edge (a5);
\path (a5) edge (a6);
\path (a6) edge (a7);
\path (a7) edge (a8);
\path (a8) edge (a9);
\path (a9) edge (a10);
\path (b5) edge (a6);
\path (c4) edge (b5);
\path (c2) edge (c3);
\path (c3) edge (c4);
\path (c4) edge (c5);
\path (c5) edge (N11a);
\path (N11a) edge (N12a);
\path (N11a) edge (N12b);
\path (N12a) edge (N13a);
\path (N12b) edge (N13a);
\path (N13a) edge (N11a);
\path (N11a) edge (c7);
\path (N12b) edge (d1);
\path (d1) edge (d2);
\path (e1) edge (e2);
\path (e2) edge (e3);
\path (e3) edge (N13a);
\path (e3) edge (N21);
\path (N13a) edge (e4);
\path (e4) edge (e5);
\path (N21) edge (N22);
\path (N22) edge (N23);
\path (N23) edge (N24);
\path (N24) edge (N21);
\path (N23) edge (f1);
\path (f1) edge (f2);
\path (f2) edge (f3);
\end{tikzpicture}}

\bigskip

\begin{minipage}{0.2\textwidth}
\begin{align*}
\MG &= 12 \\ 
\PG &= 9 \\
\EG &= 4 \\ 
\OG &= 3 \\
\ZG &= 1 \\
\BG &= 2 \\
\lambdaG &= 1
\end{align*}
\end{minipage}
\hfill
\begin{minipage}{0.7\textwidth}
\begin{center}
$\omegaG(\ell,r)$

\medskip
\begin{tabular}{c|ccccccccc}
$\ell$ $\backslash$ $r$ & 1 & 2 & 3 & 4 & 5 & 6 & 7 & 8 & $\cdots$ \\
\hline
1 & \phantom{0}3 &           -- &           -- &           -- &           -- &           -- &           -- & -- & $\cdots$ \\
2 & \phantom{0}6 & \phantom{0}4 &           -- &           -- &           -- &           -- &           -- & -- & $\cdots$ \\
3 & \phantom{0}7 & \phantom{0}7 & \phantom{0}5 &           -- &           -- &           -- &           -- & -- & $\cdots$ \\
4 & \phantom{0}8 & \phantom{0}8 & \phantom{0}8 & \phantom{0}6 &           -- &           -- &           -- & -- & $\cdots$ \\
5 & \phantom{0}8 & \phantom{0}8 & \phantom{0}8 & \phantom{0}7 & \phantom{0}7 &           -- &           -- & -- & $\cdots$ \\
6 & \phantom{0}8 & \phantom{0}8 & \phantom{0}8 & \phantom{0}8 & \phantom{0}8 & \phantom{0}8 &           -- & -- & $\cdots$ \\
7 & \phantom{0}9 & \phantom{0}9 & \phantom{0}9 & \phantom{0}9 & \phantom{0}9 & \phantom{0}9 & \phantom{0}9 & -- & $\cdots$ \\
8 &           10 &           10 &           10 &           10 &           10 &           10 &           10 & 10 & $\cdots$ \\
$\vdots$ & $\vdots$ & $\vdots$ & $\vdots$ & $\vdots$ & $\vdots$ & $\vdots$ & $\vdots$ & $\vdots$ & 
\end{tabular}
\end{center}
\end{minipage}
\end{center}

\caption{Graph $G$ and its structural parameters.}
\label{fig:G-param}
\end{figure}

\begin{definition}
\label{def:PG}\label{def:EG}\label{def:OG}
Let $G = (V,E)$ be a digraph.
Recall that a walk in $G$ is \emph{pleasant,} if all its vertices belong to trivial strongly connected components.
A walk in $G$ is \emph{winding,} if all its vertices belong to a single nontrivial strongly connected component of $G$.

Let $K$ be a nontrivial strongly connected component of $G$.
A path $v_0 \rightarrow v_1 \rightarrow \dots \rightarrow v_\ell$ in $G$ is called an \emph{entryway} to $K$ if $v_0 \rightarrow v_1 \rightarrow \dots \rightarrow v_{\ell-1}$ is a pleasant path and $v_\ell \in K$.
Analogously, $v_0 \rightarrow v_1 \rightarrow \dots \rightarrow v_\ell$ is called an \emph{outlet} from $K$ if $v_0 \in K$ and $v_1 \rightarrow v_2 \rightarrow \dots \rightarrow v_\ell$ is a pleasant path.

Denote by $\PG$, $\EG$ and $\OG$ the length of the longest pleasant path, entryway, and outlet in $G$, respectively.
If there is no finite upper bound on the length of pleasant paths, entryways, or outlets in $G$, then define $\PG := \infty$, $\EG := \infty$, $\OG := \infty$, respectively.
If there is no pleasant path, entryway, or outlet in $G$, then let $\PG := -\infty$, $\EG := -\infty$, $\OG := -\infty$, respectively.
\end{definition}

\begin{example}
In the graph $G$ of Figure~\ref{fig:G-param},
the longest pleasant path is $p_0 \rightarrow p_1 \rightarrow \dots \rightarrow p_9$,
the longest entryway is $e_0 \rightarrow e_1 \rightarrow e_2 \rightarrow e_3 \rightarrow e_4$,
and the longest outlet is $o_0 \rightarrow o_1 \rightarrow o_2 \rightarrow o_3$.
Therefore, $\PG = 9$, $\EG = 4$, $\OG = 3$.
\end{example}

\begin{lemma}\label{lem:M and P}
If a digraph $G$ satisfies the identity $t \approx t'$ for $t, t' \in B_n$, $t \neq t'$, then $\MG \divides \Mtt$ and $\PG < \Htt$. 
\end{lemma}

\begin{proof}
This follows immediately from
Lemmata~\LEMfoureight{} and \LEMfoureleven.
\end{proof}

\begin{lemma}
\label{lem:L}
Let $t, t' \in B_n$, $t \neq t'$, and
let $G$ be a digraph such $\GA{G}$ satisfies the identity $t \approx t'$.
Then $\EG \leq \Ltt + 1$.
\end{lemma}

\begin{proof}
Denote $H := \Htt$, $L := \Ltt$.
Suppose, to the contrary, that there is an entryway $W \colon v_0 \rightarrow v_1 \rightarrow \dots \rightarrow v_k$, where $k > L + 1$.
Then $v_k$ belongs to a nontrivial strongly connected component $K$ and the other vertices of $W$ belong to trivial strongly connected components.
Extending $W$, if necessary, with vertices of $K$, we obtain a walk $v_0 \rightarrow v_1 \rightarrow \dots \rightarrow v_H$, and Lemma~\LEMfourfour{} implies that $v_{L+1}$ belongs to a nontrivial strongly connected component.
This is a contradiction.
\end{proof}

\begin{definition}
\label{def:Y}
Let $t, t' \in B_n$, $t \neq t'$, and denote $T := G(t)$, $T' := G(t')$.
Let $\Ytt$ be the largest integer $m$ such that for all $x_i \in X_n$,
\[
\bigl( h(T_{x_i}) \leq m \vee h(T'_{x_i}) \leq m \bigr) \implies T_{x_i} = T'_{x_i}
\]
In other words, the rooted induced subtrees of $T$ and $T'$ of height at most $\Ytt$ are identical.
Note that $-1 \leq \Ytt < \Htt$, and the equality $\Ytt = -1$ holds if and only if
$T$ and $T'$ have different sets of leaves.
\end{definition}

\begin{example}
Figure~\ref{fig:ex:YZ} shows two DFS trees corresponding to certain terms $t, t' \in B_{14}$.
It is easy to verify that $\Ytt = 3$: all subtrees of height at most $3$ are identical in the two trees, but the subtrees rooted at $x_3$ are distinct and have height $4$.
\end{example}

\begin{figure}
\begin{center}
\tikzset{every node/.style={circle,draw,inner sep=1.5,fill=black}, every path/.style={->,>=stealth,thick}}
\scalebox{0.66}{
\begin{tikzpicture}[scale=1, baseline=(x1)]
\node[label=below:{$x_1$}] (x1) at (0,0) {};
\node[label=left:{$x_2$}] (x2) at (0,1) {};
\node[label=left:{$x_3$}] (x3) at (0,2) {};
\node[label=left:{$x_4$}] (x4) at (0,3) {};
\node[label=left:{$x_5$}] (x5) at (-1,4) {};
\node[label=left:{$x_6$}] (x6) at (-2,5) {};
\node[label=left:{$x_7$}] (x7) at (-1,5) {};
\node[label=left:{$x_8$}] (x8) at (0,4) {};
\node[label=left:{$x_9$}] (x9) at (0,5) {};
\node[label=left:{$x_{10}$}] (x10) at (0,6) {};
\node[label=right:{$x_{11}$}] (x11) at (1,2) {};
\node[label=right:{$x_{12}$}] (x12) at (1,3) {};
\node[label=right:{$x_{13}$}] (x13) at (1,4) {};
\node[label=right:{$x_{14}$}] (x14) at (2,3) {};
\node[rectangle,draw=none,fill=none] () [below of=x1] {\Large $G(t)$};
\path (x1) edge (x2);
\path (x2) edge (x3);
\path (x3) edge (x4);
\path (x4) edge (x5);
\path (x5) edge (x6);
\path (x5) edge (x7);
\path (x4) edge (x8);
\path (x8) edge (x9);
\path (x9) edge (x10);
\path (x2) edge (x11);
\path (x11) edge (x12);
\path (x12) edge (x13);
\path (x11) edge (x14);
\end{tikzpicture}
}
\qquad
\scalebox{0.66}{
\begin{tikzpicture}[scale=1, baseline=(x1)]
\node[label=below:{$x_1$}] (x1) at (0,0) {};
\node[label=left:{$x_2$}] (x2) at (0,1) {};
\node[label=left:{$x_3$}] (x3) at (0,2) {};
\node[label=left:{$x_4$}] (x4) at (0,3) {};
\node[label=left:{$x_5$}] (x5) at (-1,4) {};
\node[label=left:{$x_6$}] (x6) at (-2,5) {};
\node[label=left:{$x_7$}] (x7) at (-1,5) {};
\node[label=left:{$x_8$}] (x8) at (0,4) {};
\node[label=left:{$x_9$}] (x9) at (0,5) {};
\node[label=left:{$x_{10}$}] (x10) at (0,6) {};
\node[label=right:{$x_{11}$}] (x11) at (1,3) {};
\node[label=right:{$x_{12}$}] (x12) at (1,4) {};
\node[label=right:{$x_{13}$}] (x13) at (1,5) {};
\node[label=right:{$x_{14}$}] (x14) at (2,4) {};
\node[rectangle,draw=none,fill=none] () [below of=x1] {\Large $G(t')$};
\path (x1) edge (x2);
\path (x2) edge (x3);
\path (x3) edge (x4);
\path (x4) edge (x5);
\path (x5) edge (x6);
\path (x5) edge (x7);
\path (x4) edge (x8);
\path (x8) edge (x9);
\path (x9) edge (x10);
\path (x3) edge (x11);
\path (x11) edge (x12);
\path (x12) edge (x13);
\path (x11) edge (x14);
\end{tikzpicture}
}
\end{center}
\caption{DFS trees with $\Ytt = 3$, $\Ztt = 2$.}
\label{fig:ex:YZ}
\end{figure}
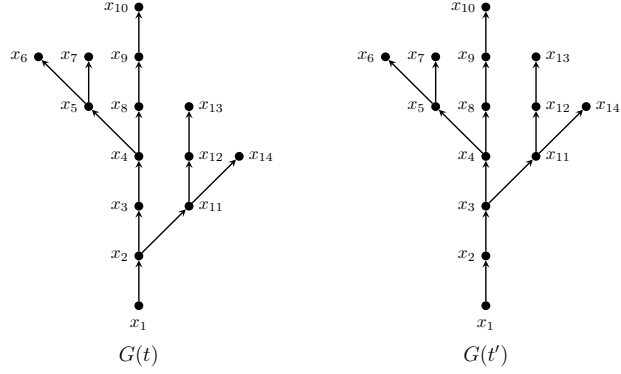

\begin{lemma}
\label{lem:Y}
Let $t, t' \in B_n$, $t \neq t'$, and
let $G$ be a digraph such that $\GA{G}$ satisfies the identity $t \approx t'$.
Then $\OG \leq \Ytt + 1$.
\end{lemma}

\begin{proof}
Denote $Y := \Ytt$.
By the definition of $Y$, there exists $x_d \in X_n$ such that $T_{x_d} \neq T'_{x_d}$ and $h(T_{x_d}) = Y + 1 \leq h(T'_{x_d})$ or $h(T'_{x_d}) = Y + 1 \leq h(T_{x_d})$. We may assume, by changing the roles of $t$ and $t'$ if necessary, that $h(T_{x_d}) = Y + 1 \leq h(T'_{x_d})$.

By the definition of a DFS tree,
$V(T_{x_d}) = X_{[d,e]}$ and $V(T'_{x_d}) = X_{[d,e']}$ for some $e, e' \in \nset{n}$.
Assume that $\Nout[T]{x_d} = \{x_{i_1}, x_{i_2}, \dots, x_{i_\ell}\}$ with $d + 1 = i_1  < i_2 < \dots < i_\ell$; hence $V(T_{x_{i_j}}) = X_{[i_j, i_{j+1} - 1]}$ for $1 \leq j \leq \ell - 1$ and $V(T_{x_{i_\ell}}) = X_{[i_\ell, e]}$.
For all $x_{i_j} \in \Nout[T]{x_d}$ it holds that $h(T_{x_{i_j}}) \leq h(T_{x_d}) - 1 = Y$; hence $T_{x_{i_j}} = T'_{x_{i_j}}$ by the definition of $Y$.
For all $x_{i_j} \in \Nout[T]{x_d}$ with $i_j > e'$, we obviously have $x_{i_j} \notin V(T'_{x_d})$ and hence $x_{i_j} \notin \Nout[T']{x_d}$.
An easy inductive argument shows that $x_{i_j} \in \Nout[T']{x_d}$ for all $x_{i_j} \in \Nout[T]{x_d}$ with $i_j \leq e'$,

We must have $e \neq e'$.
(Suppose, to the contrary, that $e = e'$. Then $\Nout[T]{x_d} = \Nout[T']{x_d}$ and consequently $T_{x_d} = T'_{x_d}$, contradicting our assumptions.)
If $e < e'$, then $\Nout[T]{x_d} \subset \Nout[T']{x_d}$; in particular, $x_{e+1} \in \Nout[T']{x_d}$.
If $e > e'$, then $\Nout[T]{x_d} \supset \Nout[T']{x_d}$; in particular, $x_{e'+1} \in \Nout[T]{x_d}$.

Suppose, to the contrary, that $G$ has an outlet $W \colon v_0 \rightarrow v_1 \rightarrow \dots \rightarrow v_k$ with $k > Y + 1$.
Then $v_0$ belongs to a nontrivial strongly connected component $K$ and the remaining vertices of $W$ belong to trivial strongly connected components.
In particular, there exists a cycle $C$ in $K$ to which $v_0$ belongs.

Consider first the case when $e < e'$.
Let $W' \colon v_1 \rightarrow \dots \rightarrow v_k$, let $x_p$ be the parent of $x_d$ in $T$, and let $\varphi \colon X_n \to V(G)$ be the collapsing map of $(T,x_d)$ on $(C,W')$ satisfying $\varphi(x_p) = v_0$.
By Proposition~\PROPtwoone, $\varphi$ is a homomorphism of $T'$ into $G$.
Since $(x_d,x_{e+1})$ is an edge of $T'$, we have the edge $(\varphi(x_d), \varphi(x_{e+1})) \in E(G)$.
Since $\varphi(x_d) = v_1$ and $\varphi(x_{e+1})$ belongs to $C$, this implies that $v_1$ belongs to the strongly connected component $K$, a contradiction.

The case when $e > e'$ is treated similarly.
Let $W' \colon v_1 \rightarrow \dots \rightarrow v_k$, let $x_{p'}$ be the parent of $x_d$ in $T'$, and let $\varphi \colon X_n \to V(G)$ be the collapsing map of $(T',x_d)$ on $(C,W')$ satisfying $\varphi(x_{p'}) = v_0$.
Note that in this case $h(T'_{x_d}) = h(T_{x_d}) = Y+1 < k$, so it is indeed possible to collapse $T'_{x_d}$ on $v_1 \rightarrow \dots \rightarrow v_k$.
A similar argument as above now shows that $(\varphi(x_d), \varphi(x_{e'+1})) \in E(G)$, which implies that $v_1$ belongs to the strongly connected component $K$, a contradiction.
\end{proof}

\begin{definition}
\label{def:Z}
Let $t, t' \in B_n$, $t \neq t'$, and denote $T := G(t)$, $T' := G(t')$.
Let $\Ztt$ be the smallest nonnegative number $m$ such that there exists $x_i \in X_n$ with $T_{x_i} = T'_{x_i}$, $h(T_{x_i}) = h(T'_{x_i}) = m$, and
$x_i$ has distinct parents in $T$ and $T'$.
Such a number $m$ always exists (see Lemma~\ref{lem:subtrees-parents} below) and it must clearly be smaller than the heights of $T$ and $T'$.
Hence $0 \leq \Ztt < \Htt$.
\end{definition}

\begin{example}
For the DFS trees of Figure~\ref{fig:ex:YZ}, it holds that $\Ztt = 2$, as witnessed by the subtrees rooted at $x_{11}$.
\end{example}

The next lemma shows that the parameter $\Ztt$ is well defined:
for distinct DFS trees $T$ and $T'$ of size $n$, there always exists a vertex $x_i \in X_n$ such that $T_{x_i} = T'_{x_i}$ and $x_i$ has distinct parents in $T$ and $T'$.

\begin{lemma}
\label{lem:subtrees-parents}
Let $T$ and $T'$ be DFS trees of size $n$.
Assume that for all $x_i \in X_n \setminus \{x_1\}$, it holds that if $T_{x_i} = T'_{x_i}$ then $x_i$ has the same parent in $T$ and in $T'$.
Then $T = T'$.
\end{lemma}

\begin{proof}
We proceed by induction on $n$.
The statement obviously holds for $n = 1$ and $n = 2$.
Assume that the statement holds for all DFS trees of size $k$.
Let $T$ and $T'$ be DFS trees of size $k+1$ satisfying the condition
that for all $x_i \in X_{k+1} \setminus \{x_1\}$, if $T_{x_i} = T'_{x_i}$ then $x_i$ has the same parent in $T$ and $T'$.

Since $x_{k+1}$ is a leaf in both $T$ and $T'$, we have $T_{x_{k+1}} = T'_{x_{k+1}}$; hence $x_{k+1}$ has the same parent in $T$ and $T'$, say $x_p$.
Consider $\overline{T} := T \setminus \{x_{k+1}\}$, $\overline{T}' := T' \setminus \{x_{k+1}\}$.
Clearly $\overline{T}$ and $\overline{T}'$ are DFS trees of size $k$, and $T = \overline{T} + (x_p, x_{k+1})$ and $T' = \overline{T}' + (x_p, x_{k+1})$ (where the notation $\overline{T} + (x_p, x_{k+1})$ stands for adjoining a new vertex $x_{k+1}$ and a new edge $(x_p, x_{k+1})$ to $\overline{T}$).
Let $x_i \in X_k$ and assume that $\overline{T}_{x_i} = \overline{T}'_{x_i}$.
If $x_p \notin V(\overline{T}_{x_i}) = V(\overline{T}'_{x_i})$, then $T_{x_i} = \overline{T}_{x_i} = \overline{T}'_{x_i} = T'_{x_i}$.
If $x_p \in V(\overline{T}_{x_i}) = V(\overline{T}'_{x_i})$, then $T_{x_i} = \overline{T}_{x_i} + (x_p, x_{k+1}) = \overline{T}'_{x_i} + (x_p, x_{k+1}) = T'_{x_i}$.
In either case, our assumption on $T$ and $T'$ implies that $x_i$ has the same parent in $T$ and $T'$ and hence also in $\overline{T}$ and $\overline{T}'$.
Consequently, $\overline{T}$ and $\overline{T}'$ satisfy the condition of the inductive hypothesis, so $\overline{T} = \overline{T}'$.
Therefore, $T = \overline{T} + (x_p, x_{k+1}) = \overline{T}' + (x_p, x_{k+1}) = T'$.
\end{proof}

\begin{definition}
\label{def:ZG}
For a digraph $G$,
let $\ZG$ be the largest nonnegative integer $m$ such that there exist a strongly connected component $K$ of $G$ that is a whirl, a block $B$ of $K$, vertices $u, w \in B$ and a walk $u \rightarrow v_0 \rightarrow v_1 \rightarrow \dots \rightarrow v_m$ but $(w,v_0) \notin E(G)$.
If there is no finite upper bound on such numbers $m$, then define $\ZG := \infty$.
If no such number $m$ exists, then define $\ZG := -\infty$.
\end{definition}

\begin{example}
In the graph $G$ of Figure~\ref{fig:G-param}, vertices $u$ and $w$ belong to the same block of a whirl.
The path $u \rightarrow z_0 \rightarrow z_1$ and the non\hyp{}edge $(w, z_0)$ witness that $\ZG = 1$.
\end{example}

\begin{lemma}
\label{lem:Z}
Let $t, t' \in B_n$, $t \neq t'$, and
let $G$ be a digraph such that $\GA{G}$ satisfies the identity $t \approx t'$.
Assume that $u$ and $w$ are vertices belonging to the same block of a nontrivial strongly connected component.
If $u \rightarrow v_0 \rightarrow v_1 \rightarrow \dots \rightarrow v_{\Ztt}$ is a walk in $G$, then $w \rightarrow v_0$ is an edge.
Consequently, $\ZG < \Ztt$.
\end{lemma}

\begin{proof}
Denote $M := \Mtt$, $Z := \Ztt$.
By Lemma~\LEMfoureight, the strongly connected component $K$ containing $u$ and $w$ is an $m$\hyp{}whirl for some divisor $m$ of $M$; let $B_a$ be the block containing $u$ and $w$.
By the definition of $Z$, there exists $x_d \in X_n$ such that $T_{x_d} = T'_{x_d}$, $h(T_{x_d}) = Z$, and the parent $x_p$ of $x_d$ in $T$ is distinct from the parent $x_q$ of $x_d$ in $T'$.
Observe that $d_T(x_p) = d_T(x_d) - 1 \equiv d_{T'}(x_d) - 1 = d_{T'}(x_q) \equiv d_T(x_q) \pmod{M}$; hence also $d_T(x_p) \equiv d_T(x_q) \pmod{m}$.
Let $C$ be a cycle of length $m$ in $K$ containing the vertex $u$.
Let $W \colon v_0 \rightarrow v_1 \rightarrow \dots \rightarrow v_Z$, and let $\varphi \colon X_n \to V(G)$ be the collapsing map of $(T,x_d)$ on $(C,W)$ with $\varphi(x_p) = u$ (also $\varphi(x_q) = u$).
Let $\psi \colon X_n \to V(G)$ be the map that coincides with $\varphi$ everywhere except at $x_q$ and satisfies $\psi(x_q) = w$.
Moreover, since $T_{x_d} = T'_{x_d}$, the vertex $x_q$ lies outside of $T_{x_d}$ and so do its children in $T$ (because $x_d$ is not a child of $x_q$ in $T$) and its parent in $T$
(because if the parent of $x_q$ lay in $T_{x_d}$, then so would $x_q$, as $T_{x_d}$ is closed under descendants).
Therefore the images of these vertices under $\varphi$ lie in $K$ (actually in $C$).
Since $u$ and $w$ belong to the same block $B_a$, the inneighbours (outneighbours, resp.)\ of $u$ and $w$ within $K$ are the same.
Consequently, $\psi$ is a homomorphism of $T$ into $G$, so, by Proposition~\PROPtwoone, $\psi$ is a homomorphism of $T'$ into $G$; hence $(w, v_0) = (\psi(x_q), \psi(x_d)) \in E(G)$.
\end{proof}

\begin{definition}
\label{def:BG}
For a digraph $G$,
let $\BG$ be the largest integer $m$ such that there exist a walk $v_0 \rightarrow v_1 \rightarrow \dots \rightarrow v_m$ and edges $v_m \rightarrow v_{m+1}$, $v_m \rightarrow v'_{m+1}$ such that $v_{m+1}$ and $v'_{m+1}$ belong to distinct nontrivial strongly connected components.
If there is no finite upper bound on such numbers $m$, then define $\BG := \infty$.
If no such number $m$ exists, then define $\BG := -\infty$.
\end{definition}

\begin{example}
In the graph $G$ of Figure~\ref{fig:G-param}, the path $b_0 \rightarrow b_1 \rightarrow b_2$ and the edges $b_2 \rightarrow v$ and $b_2 \rightarrow v'$ witness that $\BG = 2$.
\end{example}

\begin{lemma}
\label{lem:L+1}
Let $t, t' \in B_n$, $t \neq t'$, and
let $G$ be a digraph such that $\GA{G}$ satisfies the identity $t \approx t'$.
Denote $L := \Ltt$.
If $v_0 \rightarrow v_1 \rightarrow \dots \rightarrow v_{L+1}$ is a walk and $v_L \rightarrow v'_{L+1}$ is an edge in $G$ such that $v_{L+1}$ and $v'_{L+1}$ belong to nontrivial strongly connected components $K$ and $K'$, respectively, then $K = K'$.
Consequently, $\BG < \Ltt$.
\end{lemma}

\begin{proof}
Denote $H := \Htt$, $L := \Ltt$.
Using the given walks and vertices of $K$ and $K'$, we can build walks $v_0 \rightarrow \dots \rightarrow v_H$ and $v_L \rightarrow v'_{L+1} \rightarrow \dots \rightarrow v'_H$.
By Lemma~\ref{lem:s}, $v_{L+1}$ and $v'_{L+1}$ belong to the same strongly connected component, i.e., $K = K'$.
\end{proof}

\begin{definition}
\label{def:omega}
Let $t, t' \in B_n$, $t \neq t'$, and denote $T := G(t)$, $T' := G(t')$.
Let
\begin{align*}
\Deltatt & := \{x \in X_n \mid T_x \neq T'_x\}, \\
\OMEGAtt & :=
\bigl\{ (d_T(x), h(T_x)), (d_{T'}(x), h(T'_x)) \bigm| x \in \Deltatt \bigr\}, \\
\xi_{t,t'} & := \min \{ d + h \mid (d,h) \in \OMEGAtt \},
\end{align*}
and define the map $\omegatt \colon \IN \to \IN$ by the rule
\[
\omegatt(r) :=
\begin{cases}
\min \{ d + h \mid \text{$(d,h) \in \OMEGAtt$ and $d \leq r$} \}, & \text{if $r < \xi_{t,t'}$,} \\
\xi_{t,t'}, & \text{if $r \geq \xi_{t,t'}$.}
\end{cases}
\]
Note that $\omegatt(0) = \Htt$ and $\omegatt(r) > \Ltt$ for all $r \in \IN$.
Moreover, $\omegatt$ is a nonincreasing function, and
we may specify $\omegatt$ by writing down the first few values of $\omegatt$ until $\xi_{t,t'}$ is reached.
\end{definition}

\begin{example}
\label{ex:omega-lambda}
Figure~\ref{fig:ex:omega-lambda} shows two DFS trees corresponding to certain terms $t, t' \in B_{20}$.
Note that $\Ltt = 2$.
It is easy to verify that
\begin{align*}
\Deltatt &= \{ x_1, x_2, x_3, x_4, x_5, x_7, x_8, x_9, x_{10}, x_{11}, x_{16}, x_{17}, x_{18}\},
\\
\OMEGAtt &=
\{
\underbrace{(0,7), (0,6)}_{x_1},
\underbrace{(1,3), (1,4)}_{x_2},
\underbrace{(2,2), (2,3)}_{x_3},
\underbrace{(3,1), (3,2)}_{x_4},
\underbrace{(4,0), (4,1)}_{x_5},
\\ & \phantom{{}= \{}
\underbrace{(1,6), (1,5)}_{x_7},
\underbrace{(2,5), (2,4)}_{x_8},
\underbrace{(3,4), (3,3)}_{x_9},
\underbrace{(4,3), (4,2)}_{x_{10}},
\underbrace{(5,2), (5,0)}_{x_{11}},
\\ & \phantom{{}= \{}
\underbrace{(2,3), (2,2)}_{x_{16}},
\underbrace{(3,0), (3,1)}_{x_{17}},
\underbrace{(3,2), (4,0)}_{x_{18}} \},
\\
\xi_{t,t'} & = 3,
\end{align*}
whence $\omegatt \colon \IN \to \IN$ is the map $0 \mapsto 6$, $1 \mapsto 4$, $2 \mapsto 4$, $i \mapsto 3$ for $i \geq 3$,
or, using the shorthand, $\omegatt = (6,4,4,3,\dots)$.
\end{example}

\begin{figure}
\begin{center}
\tikzset{every node/.style={circle,draw,inner sep=1.5,fill=black}, every path/.style={->,>=stealth,thick}}
\scalebox{0.66}{
\begin{tikzpicture}[scale=1, baseline=(x1)]
\node[label=below:{$x_1$}] (x1) at (0,0) {};
\node[label=left:{$x_2$}] (x2) at (-1,1) {};
\node[label=left:{$x_3$}] (x3) at (-2,2) {};
\node[label=left:{$x_4$}] (x4) at (-2,3) {};
\node[label=left:{$x_5$}] (x5) at (-2.5,4) {};
\node[label=right:{$x_6$}] (x6) at (-1.5,4) {};
\node[label=left:{$x_7$}] (x7) at (0,1) {};
\node[label=left:{$x_8$}] (x8) at (0,2) {};
\node[label=left:{$x_9$}] (x9) at (0,3) {};
\node[label=left:{$x_{10}$}] (x10) at (0,4) {};
\node[label=left:{$x_{11}$}] (x11) at (0,5) {};
\node[label=left:{$x_{12}$}] (x12) at (0,6) {};
\node[label=left:{$x_{13}$}] (x13) at (0,7) {};
\node[label=right:{$x_{14}$}] (x14) at (1,3) {};
\node[label=right:{$x_{15}$}] (x15) at (1,4) {};
\node[label=right:{$x_{16}$}] (x16) at (1.5,2) {};
\node[label=right:{$x_{17}$}] (x17) at (2,3) {};
\node[label=right:{$x_{18}$}] (x18) at (3,3) {};
\node[label=right:{$x_{19}$}] (x19) at (3,4) {};
\node[label=right:{$x_{20}$}] (x20) at (3,5) {};
\node[rectangle,draw=none,fill=none] () [below of=x1] {\Large $G(t)$};
\path (x1) edge (x2);
\path (x2) edge (x3);
\path (x3) edge (x4);
\path (x4) edge (x5);
\path (x4) edge (x6);
\path (x1) edge (x7);
\path (x7) edge (x8);
\path (x8) edge (x9);
\path (x9) edge (x10);
\path (x10) edge (x11);
\path (x11) edge (x12);
\path (x12) edge (x13);
\path (x8) edge (x14);
\path (x14) edge (x15);
\path (x7) edge (x16);
\path (x16) edge (x17);
\path (x16) edge (x18);
\path (x18) edge (x19);
\path (x19) edge (x20);
\end{tikzpicture}
}
\qquad
\scalebox{0.66}{
\begin{tikzpicture}[scale=1, baseline=(x1)]
\node[label=below:{$x_1$}] (x1) at (0,0) {};
\node[label=left:{$x_2$}] (x2) at (-1,1) {};
\node[label=left:{$x_3$}] (x3) at (-2,2) {};
\node[label=left:{$x_4$}] (x4) at (-2,3) {};
\node[label=left:{$x_5$}] (x5) at (-2,4) {};
\node[label=left:{$x_6$}] (x6) at (-2,5) {};
\node[label=left:{$x_7$}] (x7) at (0,1) {};
\node[label=left:{$x_8$}] (x8) at (0,2) {};
\node[label=left:{$x_9$}] (x9) at (0,3) {};
\node[label=left:{$x_{10}$}] (x10) at (0,4) {};
\node[label=left:{$x_{11}$}] (x11) at (-1,5) {};
\node[label=left:{$x_{12}$}] (x12) at (0,5) {};
\node[label=left:{$x_{13}$}] (x13) at (0,6) {};
\node[label=right:{$x_{14}$}] (x14) at (1,4) {};
\node[label=right:{$x_{15}$}] (x15) at (1,5) {};
\node[label=right:{$x_{16}$}] (x16) at (1.5,2) {};
\node[label=right:{$x_{17}$}] (x17) at (2,3) {};
\node[label=right:{$x_{18}$}] (x18) at (2,4) {};
\node[label=right:{$x_{19}$}] (x19) at (3,3) {};
\node[label=right:{$x_{20}$}] (x20) at (3,4) {};
\node[rectangle,draw=none,fill=none] () [below of=x1] {\Large $G(t')$};
\path (x1) edge (x2);
\path (x2) edge (x3);
\path (x3) edge (x4);
\path (x4) edge (x5);
\path (x5) edge (x6);
\path (x1) edge (x7);
\path (x7) edge (x8);
\path (x8) edge (x9);
\path (x9) edge (x10);
\path (x10) edge (x11);
\path (x10) edge (x12);
\path (x12) edge (x13);
\path (x9) edge (x14);
\path (x14) edge (x15);
\path (x7) edge (x16);
\path (x16) edge (x17);
\path (x17) edge (x18);
\path (x16) edge (x19);
\path (x19) edge (x20);
\end{tikzpicture}
}
\end{center}
\caption{DFS trees with
$\omegatt = (6, 4, 4, 3, \ldots)$
and $\lambdatt = 1$.}
\label{fig:ex:omega-lambda}
\end{figure}
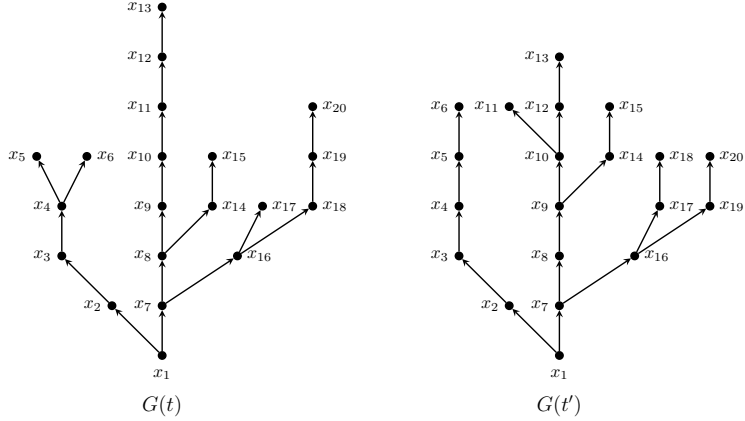

\begin{definition}
\label{def:omegaG}
Let $G$ be a digraph.
For $\ell, r \in \IN$ with $\ell \geq r \geq 1$,
let $\omegaG(\ell,r)$ be the largest integer $m$ such that there exist a walk $v_0 \rightarrow v_1 \rightarrow \dots \rightarrow v_\ell$, where $v_\ell$ belongs to a nontrivial strongly connected component, and a walk $v_{r-1} \rightarrow v'_r \rightarrow v'_{r+1} \rightarrow \dots \rightarrow v'_m$ such that $v'_\ell$ belongs to a trivial strongly connected component.
If there is no finite upper bound on such numbers $m$, then define $\omegaG(\ell,r) := \infty$.
If no such number $m$ exists, then define $\omegaG(\ell,r) := -\infty$.
Note that $\omegaG(\ell,r) \geq \ell + \OG - 1$ whenever $\OG \geq 1$ (if $o_0 \rightarrow o_1 \rightarrow \dots \rightarrow o_{\OG}$ is an outlet of length $\OG \geq 1$, then consider a walk $v_0 \rightarrow v_1 \rightarrow \dots \rightarrow v_\ell$ going around the strongly connected component of $o_0$ so that $v_{\ell-1} = o_0$ and the walk $v_{r-1} \rightarrow \dots \rightarrow v_{\ell-1} \rightarrow o_1 \rightarrow \dots \rightarrow o_{\OG}$).
\end{definition}

\begin{example}
It is not difficult to verify that for the graph $G$ of Figure~\ref{fig:G-param}, the parameter $\omegaG(\ell,r)$ has the value presented in the table in Figure~\ref{fig:G-param}.
For the values not shown in the table, that is, for $\ell, r \in \IN$ such that $\ell \geq 6$ and $\ell \geq r \geq 1$, it holds that $\omegaG(\ell,r) = \ell + 2$.
\end{example}

\begin{lemma}
\label{lem:omega}
Let $t, t' \in B_n$, $t \neq t'$, and
let $G$ be a digraph such that $\GA{G}$ satisfies the identity $t \approx t'$.
Denote $L := \Ltt$, $\omega := \omegatt$.
If $v_0 \rightarrow v_1 \rightarrow \dots \rightarrow v_{L+1}$ is a walk in $G$ such that $v_{L+1}$ belongs to a nontrivial strongly connected component, $r \in \{1, \dots, L+1\}$, and $v_{r-1} \rightarrow v'_r \rightarrow v'_{r+1} \rightarrow \dots \rightarrow v'_{\omega(r)}$ is a walk in $G$ \textup{(}recall that $\omega(r) \geq L + 1$\textup{)}, then $v'_{L+1}$ belongs to a nontrivial strongly connected component.
Consequently, $\omegaG(\Ltt + 1, r) < \omegatt(r)$ for all $r \in \{1, \dots, \Ltt + 1\}$.
\end{lemma}

\begin{proof}
Denote $H := \Htt$, $M := \Mtt$, $L := \Ltt$, $\omega := \omegatt$.
Let $K$ be the strongly connected component of $v_{L+1}$.
By Lemma~\LEMfoureight, $K$ is an $m$\hyp{}whirl for some divisor $m$ of $M$.
Let $B_a$ be the block of $K$ containing $v_{L+1}$, and let $B_{a-1}$ be the predecessor block of $B_a$.

If $\omega(r) \geq H$, then the claim follows immediately from Lemma~\LEMfourfour.
We can thus assume that $\omega(r) < H$.
By the definition of $\omega(r)$ and $\OMEGAtt$, there exists a vertex $x_d \in X_n$ such that $T_{x_d} \neq T'_{x_d}$, and either $d_T(x_d) \leq r$ and $d_T(x_d) + h(T_{x_d}) = \omega(r)$ or $d_{T'}(x_d) \leq r$ and $d_{T'}(x_d) + h(T'_{x_d}) = \omega(r)$;
moreover, for all $x_i \in X_n$ such that $T_{x_i} \neq T'_{x_i}$, it holds that $d_T(x_i) \leq r$ implies $d_T(x_i) + h(T_{x_i}) \geq \omega(r)$, and $d_{T'}(x_i) \leq r$ implies $d_{T'}(x_i) + h(T'_{x_i}) \geq \omega(r)$.
We may assume, by changing the roles of $t$ and $t'$ if necessary, that
$d_T(x_d) \leq r$ and $d_T(x_d) + h(T_{x_d}) = \omega(r)$.
Note that if $d_T(x_d) \leq L$ or $d_{T'}(x_d) \leq L$, then, by the definition of $L$, we have $d_T(x_d) = d_{T'}(x_d)$.
Since $d_T(x_d) \leq r \leq L + 1$,
it follows from our assumptions that
either
$d_T(x_d) = d_{T'}(x_d) \leq L + 1$ and $h(T_{x_d}) \leq h(T'_{x_d})$,
or
$d_T(x_d) = L + 1 < d_{T'}(x_d)$.

We are going to make use of the homomorphism $\varphi \colon T \to G$ that is defined as follows.
Fix an $m$\hyp{}cycle $C$ in $K$ that contains the vertex $v_{L+1}$, and let $W$ be a walk that starts with $v_0 \rightarrow v_1 \rightarrow \dots \rightarrow v_{L+1}$ and continues around $C$ until it reaches length $h(T)$.
Let $W'$ be the walk $v_{d_T(x_d)} \rightarrow \dots \rightarrow v_{r-1} \rightarrow v'_r \rightarrow v'_{r+1} \rightarrow \dots \rightarrow v'_{\omega(r)}$ if $d_T(x_d) < r$ and $v'_r \rightarrow v'_{r+1} \rightarrow \dots \rightarrow v'_{\omega(r)}$ if $d_T(x_d) = r$.
Note that $W'$ has length exactly $h(T_{x_d})$ because $d_T(x_d) + h(T_{x_d}) = \omega(r)$.
Let $\varphi \colon X_n \to V(G)$ be the collapsing map of $(T,x_d)$ on $(W,W')$.
By Proposition~\PROPtwoone, $\varphi$ is also a homomorphism of $T'$ into $G$.

We have $V(T_{x_d}) = X_{[d,e]}$ and $V(T'_{x_d}) = X_{[d,e']}$ for some $e, e' \in \nset{n}$.
Consequently $V(T_{x_d}) \subseteq V(T'_{x_d})$ (if $e \leq e'$) or $V(T'_{x_d}) \subseteq V(T_{x_d})$ (if $e' \leq e$).
We will consider several cases and subcases.

Case 1:
$V(T_{x_d}) \subsetneq V(T'_{x_d})$, i.e., $e < e'$.
Then necessarily $r = d_T(x_d) = L + 1$ and $x_{e+1} \in V(T'_{x_d}) \setminus V(T_{x_d})$;
note that $W'$ is the walk $v'_{L+1} \rightarrow \dots \rightarrow v'_{\omega(r)}$.
Let $x_p$ be the parent of $x_{e+1}$ in $T'$.
Then $d \leq p < e + 1$, so $x_p \in V(T_{x_d})$.
Moreover, since $x_{e+1}$ has different parents in $T$ and $T'$, we must have $d_T(x_{e+1}) \geq L + 1$ by the definition of $L$.
Since $\varphi \colon T' \to G$ is a homomorphism, we have $(\varphi(x_p), \varphi(x_{e+1})) \in E(G)$.
Since $x_p \in V(T_{x_d})$, we have $\varphi(x_p) \in \{v'_r, v'_{r+1}, \dots, v'_{\omega(r)}\}$; since $x_{e+1} \notin V(T_{x_d})$ and $d_T(x_{e+1}) \geq L + 1$, we have $\varphi(x_{e+1}) \in K$.
Now we can extend the walk $v_0 \rightarrow \dots \rightarrow v_L \rightarrow v'_{L+1} \rightarrow \dots \rightarrow \varphi(x_p) \rightarrow \varphi(x_{e+1})$ with vertices of $K$ so that we obtain a walk of length $H$, and Lemma~\LEMfourfour{} implies that $v'_{L+1}$ belongs to a nontrivial strongly connected component, in fact, to $K$ by Lemma~\LEMfournine.

Case 2:
$V(T_{x_d}) \supseteq V(T'_{x_d})$, i.e., $e \geq e'$.
Then $\varphi$ maps $V(T'_{x_d})$ on $W'$.

Case 2.1:
$h(T_{x_d}) < h(T'_{x_d}) =: h'$.
Let $x_d = u_0, u_1, \dots, u_{h'}$ be a longest path in $T'_{x_d}$.
Write $d_i := d_T(u_i)$ for $i \in \{0, \dots, h'\}$.
Since $h(T_{x_d}) < h(T'_{x_d})$, the sequence $d_0, d_1, \dots, d_{h'}$ cannot be strictly increasing, so there is an index $i \in \{0, \dots, h' - 1\}$ such that $d_i \geq d_{i+1}$; in fact, $d_{i+1} \geq L+1$ by the definition of $L$.
Then $(\varphi(u_i), \varphi(u_{i+1})) = (v'_{d_i}, v'_{d_{i+1}}) \in E(G)$, so $v'_{d_{i+1}} \rightarrow \dots \rightarrow v'_{d_i} \rightarrow v'_{d_{i+1}}$ is a closed walk in $G$.
It then follows easily from Lemma~\LEMfourfour{} that $v'_{L+1}$ belongs to a nontrivial strongly connected component.

Case 2.2: $h(T_{x_d}) \geq h(T'_{x_d})$.
Recall that
either
$d_T(x_d) = d_{T'}(x_d) \leq L + 1$ and $h(T_{x_d}) \leq h(T'_{x_d})$,
or
$d_T(x_d) = L + 1 < d_{T'}(x_d)$.
We consider separately these two cases.

Case 2.2.1: $d_T(x_d) = d_{T'}(x_d) \leq L + 1$ and $h(T_{x_d}) \leq h(T'_{x_d})$.
It follows from our assumptions that $h(T_{x_d}) = h(T'_{x_d})$.
If $V(T_{x_d}) \supsetneq V(T'_{x_d})$, then we can repeat the above argument with the roles of $t$ and $t'$ switched, and we will reach Case~1 and we are done.
We can now assume that $V(T_{x_d}) = V(T'_{x_d})$ (note that this holds if $d_T(x_d) = d_{T'}(x_d) \leq L$).
Observe that now the roles of $t$ and $t'$ are symmetric; we would reach this point in the argument even if $t$ and $t'$ were switched, and we may swap them if necessary.

Since $T_{x_d} \neq T'_{x_d}$, there exists an element $x_q \in V(T_{x_d})$ such that $d_T(x_q) \neq d_{T'}(x_q)$; assume that the index $q$ is the smallest possible.
Swapping the roles of $t$ and $t'$, if necessary, we may assume that $d_T(x_q) < d_{T'}(x_q)$; moreover $d_T(x_q) \geq L + 1$ by the definition of $L$.
Let $x_p$ be the parent of $x_q$ in $T'$.
Then $p < q$, so by the choice of $x_q$, we have $d_T(x_p) = d_{T'}(x_p) = d_{T'}(x_q) - 1 \geq d_T(x_q) \geq L + 1$.
Since $\varphi \colon T' \to G$ is a homomorphism, we have $(\varphi(x_p), \varphi(x_q)) = (v'_{d_p}, v'_{d_q}) \in E(G)$, where $d_p := d_T(x_p)$, $d_q := d_T(x_q)$.
Then $v'_{d_q} \rightarrow \dots \rightarrow v'_{d_p} \rightarrow v'_{d_q}$ is a closed walk in $G$.
It then follows easily from Lemma~\LEMfourfour{} that $v'_{L+1}$ belongs to a nontrivial strongly connected component.

Case 2.2.2: $d_T(x_d) = L + 1 < d_{T'}(x_d)$.
Since $1 \leq r \leq L+1$ and $d_T(x_d) \leq r$, we have $r = L+1$ in this case; therefore $W'$ is the walk $v'_{L+1} \rightarrow \dots \rightarrow v'_{\omega(r)}$.
Let $x_p$ be the parent of $x_d$ in $T'$.
Then $p < d$, so $x_p \notin V(T_{x_d})$, and
$d_T(x_p) \equiv d_{T'}(x_p) = d_{T'}(x_d) - 1 \equiv d_T(x_d) - 1 = L \pmod{M}$.
Moreover, $d_{T'}(x_p) \geq L + 1$, so also $d_T(x_p) \geq L + 1$ by the definition of $L$, and we have $w := \varphi(x_p) \in B_{a-1}$.
Since $\varphi \colon T' \to G$ is a homomorphism, we have $(\varphi(x_p), \varphi(x_d)) = (w, v'_{L+1}) \in E(G)$.

Define homomorphisms $\psi \colon T \to G$ and $\psi' \colon T' \to G$ as follows.
Let $\psi$ be the collapsing map of $(T,x_d)$ on $(C,W')$ that maps the parent of $x_d$ in $T$ to $w$, and
let $\psi'$ be the collapsing map of $(T',x_d)$ on $(C,W')$ that maps the parent of $x_d$ in $T'$ to $w$.

Recall that we are assuming that $V(T_{x_d}) \supseteq V(T'_{x_d})$ and $h(T_{x_d}) \geq h(T'_{x_d})$.
If $V(T_{x_d}) \supsetneq V(T'_{x_d})$, then using a similar argument as in Case~1 with the homomorphism $\psi'$ in place of $\varphi$, we can find an edge from $W'$ to $K$, from which it follows that $v'_{L+1}$ belongs to a nontrivial strongly connected component.
We can thus assume that $V(T_{x_d}) = V(T'_{x_d})$.
If $h(T_{x_d}) > h(T'_{x_d})$, then using a similar argument as in Case~2.1 with the homomorphism $\psi'$ in place of $\varphi$, we can find a closed walk in $W'$, from which it follows that $v'_{L+1}$ belongs to a nontrivial strongly connected component.
We can thus assume that $h(T_{x_d}) = h(T'_{x_d})$.
Now, using a similar argument as in Case~2.2.1 with the homomorphism $\psi$ or $\psi'$ in place of $\varphi$, we can find a closed walk in $W'$, from which it again follows that $v'_{L+1}$ belongs to a nontrivial strongly connected component.
\end{proof}

\begin{definition}
\label{def:lambda}
Let $t, t' \in B_n$, $t \neq t'$, and denote $T := G(t)$, $T' := G(t')$.
Let
\[
\Lambdatt := 
\bigl\{
x \in X_n \bigm|
d_T(x) \neq d_{T'}(x), \,
\Ltt + 1 \in \{d_T(x), d_{T'}(x)\}
\bigr\},
\]
and let
\[
\lambdatt :=
\min \bigl\{ \max (h(T_x), h(T'_x)) \bigm| x \in \Lambdatt \bigr\}.
\]
Note that $\Lambdatt \neq \emptyset$ by the definition of $\Ltt$; hence $\lambdatt$ is well defined and $\lambdatt \geq 0$.
\end{definition}

\begin{example}
For the DFS trees of Figure~\ref{fig:ex:omega-lambda}, it holds that
\begin{gather*}
\Lambdatt = \{x_{14}, x_{18}, x_{19}\},
\\
\lambdatt = \min \{ \underbrace{\max(1,1)}_{x_{14}}, \underbrace{\max(2,0)}_{x_{18}}, \underbrace{\max(1,1)}_{x_{19}} \} = \min \{1, 2, 1\}  = 1.
\end{gather*}
\end{example}

\begin{definition}
\label{def:lambdaG}
Let $G$ be a digraph.
Let $\lambdaG$ be the largest integer $m$ such that there exist an entryway $u_0 \rightarrow u_1 \rightarrow \dots \rightarrow u_{\EG}$ (of maximal length $\EG$) to a nontrivial strongly connected component $K$, a vertex $w$ in $K$ with $w \rightarrow u_{\EG}$ and a walk $v_0 \rightarrow v_1 \rightarrow \dots \rightarrow v_m$ such that exactly one of the pairs $(w,v_0)$ and $(u_{\EG-1},v_0)$ is an edge and the other is not.
If there is no upper bound for such numbers $m$, then define $\lambdaG := \infty$.
If no such number $m$ exists (this holds in particular when $\EG \leq 0$), then define $\lambdaG := -\infty$.
\end{definition}

\begin{example}
In the graph $G$ of Figure~\ref{fig:G-param}, the longest entryway $e_0 \rightarrow e_1 \rightarrow e_2 \rightarrow e_3 \rightarrow e_4$, the path $\lambda_0 \rightarrow \lambda_1$, the edges $v \rightarrow e_4$ and $v \rightarrow \lambda_0$ and the nonedge $(e_3, \lambda_0)$ witness that $\lambdaG = 1$.
\end{example}

\begin{lemma}
\label{lem:shortcut}
Let $t, t' \in B_n$, $t \neq t'$, and
let $G$ be a digraph such that $\GA{G}$ satisfies the identity $t \approx t'$.
Denote $L := \Ltt$, $\lambda := \lambdatt$.
Assume that $\EG = L + 1$,
$u_0 \rightarrow u_1 \rightarrow \dots \rightarrow u_L \rightarrow u_{L+1}$ is an entryway to a nontrivial strongly connected component $K$,
$w$ is a vertex in $K$ with $w \rightarrow u_{L+1}$,
and $v_0 \rightarrow v_1 \rightarrow \dots \rightarrow v_\lambda$ is a walk in $G$.
Then $w \rightarrow v_0$ is an edge if and only if $u_L \rightarrow v_0$ is an edge.
\textup{(}See Figure~\ref{fig:shortcut}.\textup{)}
Consequently, $\lambdaG < \lambdatt$.
\end{lemma}

\begin{figure}
\begin{center}
\tikzset{every node/.style={circle,draw,inner sep=1.5,fill=black}, every path/.style={->,>=stealth,thick}}
\tikzstyle{katko}=[dash pattern=on 9pt off 3pt on 2pt off 3pt on 2pt off 3pt on 2 pt off 3pt on 10pt]
\scalebox{1}{
\begin{tikzpicture}[scale=1]
\draw[thin,fill opacity=0.2,draw opacity=0.5,draw=gray!60!black,fill=gray] (3.25,2.1) circle [x radius = 2.75, y radius = 1.25];
\node[label=south:{$u_0$}] (u0) at (0,0) {};
\node[label=south:{$u_1$}] (u1) at (1,0) {};
\node[label=south:{$u_L$}] (uL) at (2.5,0) {};
\node[label=west:{$u_{L+1}$}] (uL1) at (2.5,1.5) {};
\node[label=east:{$w$}] (w) at (4,1.5) {};
\node[label=south:{$v_0$}] (v0) at (4,0) {};
\node[label=south:{$v_1$}] (v1) at (5,0) {};
\node[label=south:{$v_\lambda$}] (vl) at (6.5,0) {};
\path (u0) edge (u1);
\path (u1) edge[katko] (uL);
\path (uL) edge (uL1);
\path (w) edge (uL1);
\path (v0) edge (v1);
\path (v1) edge[katko] (vl);
\path (w) edge[densely dotted] (v0);
\path (uL) edge[densely dotted] (v0);
\node[rectangle,draw=none,fill=none,rotate=45] (iff) at (3.6,0.4) {$\iff$};
\node[rectangle,draw=none,fill=none] (K) at (3.25,2.9) {$K$};
\end{tikzpicture}}
\end{center}
\caption{Illustration for Lemma~\ref{lem:shortcut}.}
\label{fig:shortcut}
\end{figure}
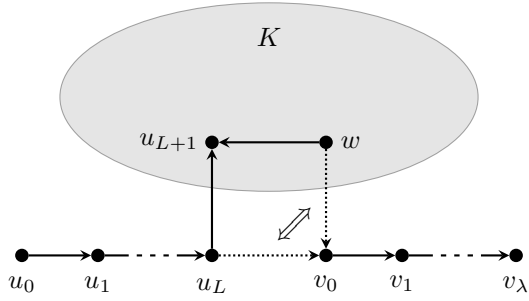

\begin{proof}
Denote $M := \Mtt$, $L := \Ltt$, $\lambda := \lambdatt$.
By the definition of $\lambda$, there exists an element $x_d \in X_n$ such that
$L + 1 \in \{d_T(x_d), d_{T'}(x_d)\}$, $d_T(x_d) \neq d_{T'}(x_d)$, and $\max(h(T_{x_d}), h(T'_{x_d})) = \lambda$.
By changing the roles of $t$ and $t'$ if necessary, we may assume that $d_T(x_d) = L + 1 < d_{T'}(x_d)$.
Let $x_p$ and $x_q$ be the parents of $x_d$ in $T$ and $T'$, respectively.
Then $d_T(x_p) = d_T(x_d) - 1 = L$,
$d_{T'}(x_q) = d_{T'}(x_d) - 1 \geq L + 1$ and $d_{T'}(x_q) = d_{T'}(x_d) - 1 \equiv d_T(x_d) - 1 = L \pmod{M}$.

Denote by $W$ the entryway $u_0 \rightarrow u_1 \rightarrow \dots \rightarrow u_L \rightarrow u_{L+1}$ and by $W'$ the walk $v_0 \rightarrow v_1 \rightarrow \dots \rightarrow v_\lambda$.
By Lemma~\LEMfoureight, $K$ is an $m$\hyp{}whirl for some divisor $m$ of $M$.
Let $C$ be an $m$\hyp{}cycle in $K$ that includes the vertices $w$ and $u_{L+1}$.

If $(u_L, v_0) \in E(G)$, then
let $W''$ be the walk that extends $W$ with vertices of $C$ to a walk of length $h(T)$, and consider the collapsing map $\varphi \colon X_n \to V(G)$ of $(T,x_d)$ to $(W'',W')$.
Observe that $\varphi(x_q) = w$.
(In order to see this, we need to verify that $x_q \notin T_{x_d}$, $d_T(x_q) \geq L + 1$ and $d_T(x_q) \equiv L \pmod{m}$.
The condition $x_q \notin T_{x_d}$ holds because $q < d$, as $x_q$ is the parent of $x_d$ in $T'$.
If $d_T(x_q) \leq L$, then $d_T(x_q) = d_{T'}(x_q)$ by the definition of $L$; hence $d_{T'}(x_q) \leq L$, which is a contradiction because we have seen that $d_{T'}(x_q) \geq L + 1$.
We have also seen that $d_{T'}(x_q) \equiv L \pmod{M}$, and $d_T(x_q) \equiv d_{T'}(x_q) \pmod{M}$ by the definition of $M$.
These imply $d_T(x_q) \equiv L \pmod{M}$, and then $d_T(x_q) \equiv L \pmod{m}$ follows, as $m \divides M$.)
By Proposition~\PROPtwoone, $\varphi$ is a homomorphism $T' \to G$, so
$(\varphi(x_q), \varphi(x_d)) = (w, v_0) \in E(G)$.

If $(w, v_0) \in E(G)$, then
let $W''$ be the walk that extends $W$ with vertices of $C$ to a walk of length $h(T')$, and consider the collapsing map $\varphi' \colon X_n \to V(G)$ of $(T',x_d)$ to $(W'',W')$.
Observe that $\varphi'(x_p) = u_L$.
(In order to see this, we need to verify that $d_{T'}(x_p) = L$.
We know that $d_T(x_p) = L$, so $d_T(x_p) = d_{T'}(x_p)$ by the definition of $L$.
From this it follows that $d_{T'}(x_p) = L$.)
By Proposition~\PROPtwoone, $\varphi'$ is a homomorphism $T \to G$, so
$(\varphi'(x_p), \varphi'(x_d)) = (u_L, v_0) \in E(G)$.
\end{proof}

\begin{remark}
Note that the walk $v_0 \rightarrow v_1 \rightarrow \dots \rightarrow v_\lambda$ in Lemma~\ref{lem:shortcut} may include vertices in the nontrivial strongly connected component $K$.
In particular, Lemma~\ref{lem:shortcut} asserts that if $G$ satisfies $t \approx t'$, $L := \Ltt$, $E_G = L + 1$, and
$u_0 \rightarrow u_1 \rightarrow \dots \rightarrow u_L \rightarrow u_{L+1}$ is an entryway,
then there is an edge $u_L \rightarrow v$ for every vertex $v$ in the block $B$ of $u_{L+1}$ in $K$.
This follows by choosing any vertex $w$ from the predecessor block of $B$ and taking $v_0 \rightarrow v_1 \rightarrow \dots \rightarrow v_\lambda$ to be any walk starting at $v$ and going around $K$ until it reaches length $\lambda$.
\end{remark}

We have established above several necessary conditions for a digraph to satisfy a bracketing identity.
We show next that these conditions are also sufficient.

\begin{theorem}
\label{thm:main-parameters}
Let $G$ be a digraph, and let $t, t' \in B_n$ with $t \neq t'$.
Then $\GA{G}$ satisfies the identity $t \approx t'$ if and only if the following conditions hold:
\begin{enumerate}[label={\upshape{(\roman*)}}]
\item\label{cond:whirl}
Every nontrivial strongly connected component of $G$ is a whirl.
\item\label{cond:no-SCC-SCC}
There is no path from a nontrivial strongly connected component of $G$ to another.
\item\label{cond:M}
$\MG \divides \Mtt$.
\item\label{cond:H}
$\PG < \Htt$.
\item\label{cond:L}
$\EG \leq \Ltt + 1$.
\item\label{cond:Y}
$\OG \leq \Ytt + 1$.
\item\label{cond:Z}
$\ZG < \Ztt$.
\item\label{cond:L+1}
$\BG < \Ltt$.
\item\label{cond:omega}
$\omegaG(\Ltt + 1, r) < \omegatt(r)$ for all $r \in \{1, \dots, \Ltt + 1\}$.
\item\label{cond:shortcut}
If $\EG = \Ltt + 1$, then $\lambdaG < \lambdatt$.
\end{enumerate}
\end{theorem}

\begin{proof}
Denote $T := G(t)$, $T' := G(t')$,
$H := \Htt$, $M := \Mtt$, $L := \Ltt$, $Y := \Ytt$, $Z := \Ztt$, $\omega := \omegatt$, $\lambda := \lambdatt$.
The necessity of the conditions is established in
Lemmata
\LEMfournine,
\ref{lem:M and P},
\ref{lem:L},
\ref{lem:Y},
\ref{lem:Z},
\ref{lem:L+1},
\ref{lem:omega},
\ref{lem:shortcut}.

For sufficiency, assume that the digraph $G = (V,E)$ and the bracketings $t, t' \in B_n$ satisfy the conditions.
In order to show that $\GA{G}$ satisfies the identity $t \approx t'$, it suffices, by Proposition~\PROPtwoone, to show that a map $\varphi \colon X_n \to V$ is a homomorphism of $T$ into $G$ if and only if it is a homomorphism of $T'$ into $G$.
So, assume that $\varphi \colon X_n \to V$ is a homomorphism of $T$ into $G$.
We need to verify that $\varphi$ is a homomorphism of $T'$ into $G$.

The image of any path in $T$ under $\varphi$ is a walk in $G$.
By conditions~\ref{cond:no-SCC-SCC}, \ref{cond:L} and \ref{cond:Y},
it is either a pleasant path,
or it comprises an entryway (of length at most $L+1$, possibly $0$) to a nontrivial strongly connected component $K$, followed by a winding walk in $K$, again followed by an outlet from $K$ (of length at most $Y+1$, possibly $0$).
Since $T$ contains a path of length $h(T) \geq H$, condition~\ref{cond:H} implies that
the image of $\varphi$ contains a vertex belonging to a nontrivial strongly connected component of $G$.

Our goal is to show that for any edge $(a,b)$ of $T'$, its image $(\varphi(a), \varphi(b))$ is an edge of $G$.
Since $T$ and $T'$ are identical up to level $L$, it holds that if $(a,b)$ is an edge of $T'$ with $d_{T'}(a) < L$, then $(a,b)$ is also an edge of $T$ and hence $(\varphi(a), \varphi(b)) \in E(G)$.
Therefore we can focus on edges $(a,b) \in E(T')$ with $d_{T'}(a) \geq L$.

Let $x_\ell \in X_n$ be an arbitrary vertex with $d_{T'}(x_\ell) = L$.
Then also $d_T(x_\ell) = L$ and $V(T_{x_\ell}) = V(T'_{x_\ell}) = X_{[\ell,\ell']}$ for some $\ell' \geq \ell$.
We will be done if we show that $(\varphi(a), \varphi(b)) \in E(G)$ holds for every edge $(a,b)$ of the rooted induced subtree $T'_{x_\ell}$.
The remainder of the proof is a case analysis.
The first case distinction is made according to which vertices of $T_{x_\ell}$, if any, are mapped to nontrivial strongly connected components.
Each case leads to several subcases.
Figure~\ref{fig:main-cases} illustrates several main cases and subcases, showing relevant parts of the tree $T$ and highlighting vertices that are mapped to nontrivial strongly connected components.

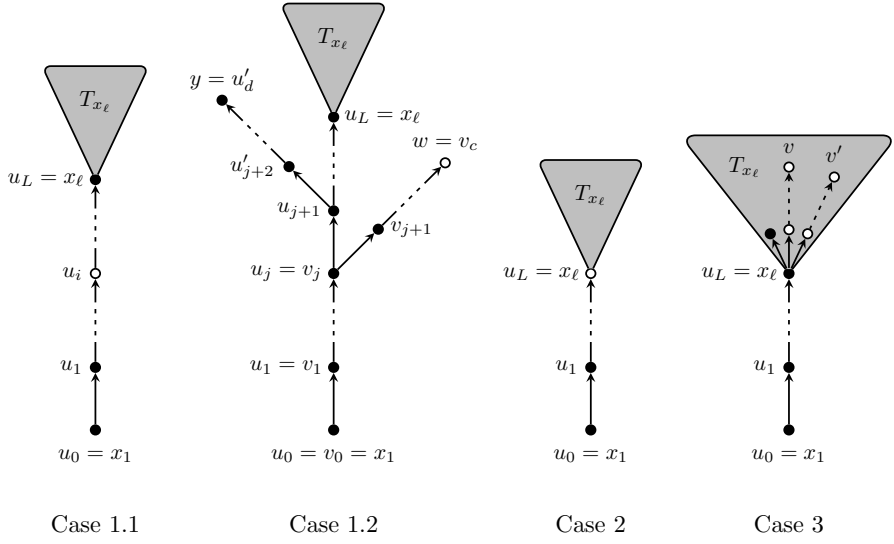
\begin{figure}
\begin{center}
\tikzset{every node/.style={circle,draw,inner sep=1.5,fill=black}, every path/.style={->,>=stealth,thick}}
\tikzstyle{katko}=[dash pattern=on 9pt off 3pt on 2pt off 3pt on 2pt off 3pt on 2 pt off 3pt on 10pt]
\scalebox{0.83}{
\begin{tikzpicture}[scale=1, baseline=(x1.center)]
\node[label={[label distance=-10pt]south:{$u_0 = x_1$}}] (u0) at (0,0) {};
\node[label=west:{$u_1$}] (u1) at (0,1) {};
\node[label=west:{$u_i$}, fill=none] (ui) at (0,2.5) {};
\node[label=west:{$u_L = x_\ell$}] (xL) at (0,4) {};
\path (u0) edge (u1);
\path (u1) edge[katko] (ui);
\path (ui) edge[katko] (xL);
\coordinate (Tl) at ($(xL) + (115:2)$);
\coordinate (Tr) at ($(xL) + (65:2)$);
\draw[rounded corners=4, -, fill=lightgray] (xL) -- (Tl) -- (Tr) -- (xL);
\node[rectangle,draw=none,fill=none] (TxL) at ($(xL) + (0,1.25)$) {$T_{x_\ell}$};
\node[rectangle,draw=none,fill=none] (Case) at (0,-1.5) {\large Case 1.1};
\end{tikzpicture}}
\,\,
\scalebox{0.83}{
\begin{tikzpicture}[scale=1, baseline=(x1.center)]
\node[label={[label distance=-21pt]south:{$u_0 = v_0 = x_1$}}] (u0) at (0,0) {};
\node[label=west:{$u_1 = v_1$}] (u1) at (0,1) {};
\node[label=west:{$u_j = v_j$}] (uj) at (0,2.5) {};
\node[label=west:{$u_{j+1}$}] (uj1) at (0,3.5) {};
\node[label=east:{$u_L = x_\ell$}] (xL) at (0,5) {};
\node[label=east:{$v_{j+1}$}] (vj1) at ($(uj)+(45:1)$) {};
\node[label={[label distance=-11pt]north:{$w = v_c$}}, fill=none] (w) at ($(vj1) + (45:1.5)$) {};
\node[label=west:{$u'_{j+2}$}] (uj2) at ($(uj1) + (135:1)$) {};
\node[label={[label distance=-11pt]north:{$y = u'_d$}}] (y) at ($(uj2) + (135:1.5)$) {};
\path (u0) edge (u1);
\path (u1) edge[katko] (uj);
\path (uj) edge (uj1);
\path (uj1) edge[katko] (xL);
\path (uj) edge (vj1);
\path (vj1) edge[katko] (w);
\path (uj1) edge (uj2);
\path (uj2) edge[katko] (y);
\coordinate (Tl) at ($(xL) + (115:2)$);
\coordinate (Tr) at ($(xL) + (65:2)$);
\draw[rounded corners=4, -, fill=lightgray] (xL) -- (Tl) -- (Tr) -- (xL);
\node[rectangle,draw=none,fill=none] (TxL) at ($(xL) + (0,1.25)$) {$T_{x_\ell}$};
\node[rectangle,draw=none,fill=none] (Case) at (0,-1.5) {\large Case 1.2};
\end{tikzpicture}}
\scalebox{0.83}{
\begin{tikzpicture}[scale=1, baseline=(x1.center)]
\node[label={[label distance=-10pt]south:{$u_0 = x_1$}}] (u0) at (0,0) {};
\node[label=west:{$u_1$}] (u1) at (0,1) {};
\node[label=west:{$u_L = x_\ell$}, fill=none] (xL) at (0,2.5) {};
\path (u0) edge (u1);
\path (u1) edge[katko] (xL);
\coordinate (Tl) at ($(xL) + (115:2)$);
\coordinate (Tr) at ($(xL) + (65:2)$);
\draw[rounded corners=4, -, fill=lightgray] (xL) -- (Tl) -- (Tr) -- (xL);
\node[rectangle,draw=none,fill=none] (TxL) at ($(xL) + (0,1.25)$) {$T_{x_\ell}$};
\node[rectangle,draw=none,fill=none] (Case) at (0,-1.5) {\large Case 2};
\end{tikzpicture}}
\,\,
\scalebox{0.83}{
\begin{tikzpicture}[scale=1, baseline=(x1.center)]
\node[label={[label distance=-10pt]south:{$u_0 = x_1$}}] (u0) at (0,0) {};
\node[label=west:{$u_1$}] (u1) at (0,1) {};
\node[label=west:{$u_L = x_\ell$}] (xL) at (0,2.5) {};
\coordinate (Tl) at ($(xL) + (128:2.8)$);
\coordinate (Tr) at ($(xL) + (52:2.8)$);
\draw[rounded corners=7, -, fill=lightgray] (xL) -- (Tl) -- (Tr) -- (xL);
\node (xr1) at ($(xL)+(115:0.7)$) {};
\node[fill=white] (xr2) at ($(xL)+(90:0.7)$) {};
\node[fill=white] (xr3) at ($(xL)+(65:0.7)$) {};
\node[label=north:{$v$}, fill=white] (xs2) at ($(xr2)+(90:1)$) {};
\node[label=north:{$v'$}, fill=white] (xs3) at ($(xr3)+(65:1)$) {};
\path (u0) edge (u1);
\path (u1) edge[katko] (xL);
\path (xL) edge (xr1);
\path (xL) edge (xr2);
\path (xL) edge (xr3);
\path (xr2) edge[katko, dash phase=7.5pt] (xs2);
\path (xr3) edge[katko, dash phase=7.5pt] (xs3);
\node[rectangle,draw=none,fill=none] (TxL) at ($(xL) + (-0.7,1.7)$) {$T_{x_\ell}$};
\node[rectangle,draw=none,fill=none] (Case) at (0,-1.5) {\large Case 3};
\end{tikzpicture}}
\end{center}
\caption{Various cases considered in the proof of Theorem~\ref{thm:main-parameters}.
Hollow vertices are mapped to a nontrivial strongly connected component of $G$.}
\label{fig:main-cases}
\end{figure}

Case 1:
Assume that $\varphi$ maps no vertex of $T_{x_\ell}$ to a nontrivial strongly connected component of $G$.
Let $x_1 =: u_0 \rightarrow u_1 \rightarrow \dots \rightarrow u_L := x_\ell$ be the path from $x_1$ to $x_\ell$ in $T'$ (equivalently, in $T$).
We make a further case distinction on whether any vertex on this path is mapped to a nontrivial strongly connected component.

Case 1.1:
Assume that there is an index $i \in \{0, \dots, L-1\}$ such that $\varphi(u_i)$ lies in a nontrivial strongly connected component of $G$.
It follows from condition~\ref{cond:Y} that $h(T_{x_\ell}) \leq Y$; hence $T_{x_\ell} = T'_{x_\ell}$ by the definition of $Y$.
Therefore $(\varphi(a),\varphi(b))$ is clearly an edge of $G$ for every edge $(a,b)$ of $T'_{x_\ell}$.

Case 1.2:
Assume that for all $i \in \{0, \dots, L-1\}$, $\varphi(u_i)$ belongs to a trivial strongly connected component.
Since the image of $\varphi$ contains a vertex belonging to a nontrivial strongly connected component of $G$, there exists an index $j \in \{0, \dots, L-1\}$ such that $T_{u_j}$ contains a vertex that is mapped by $\varphi$ to a nontrivial strongly connected component (at least $T_{x_1} = T_{u_0}$ satisfies this).
Assume that $j$ is the largest such index.
By condition~\ref{cond:L}, $T_{u_j}$ contains a vertex $w$ such that $\varphi(w)$ lies in a nontrivial strongly connected component $K$ and $c := d_T(w) \leq L+1$.
Let $x_1 =: v_0 \rightarrow v_1 \rightarrow \dots \rightarrow v_c$ be the path from $x_1$ to $w$ in $T$; note that $v_i = u_i$ for all $i \leq j$.
Then $\varphi(v_0) \rightarrow \varphi(v_1) \rightarrow \dots \rightarrow \varphi(v_c)$ is a walk in $G$.
Continuing this in a suitable way with vertices from $K$, we obtain a walk of length $L+1$ in $G$, the last vertex of which belongs to $K$.
Let then $y$ be a vertex of maximum depth in $T_{u_{j+1}}$, let $d := d_T(y)$, and consider the path $u_0 \rightarrow u_1 \rightarrow \dots \rightarrow u_{j+1} \rightarrow u'_{j+2} \rightarrow \dots \rightarrow u'_d$ from $x_1$ to $y$ in $T$.
By the choice of $j$, the walk $\varphi(u_0) \rightarrow \varphi(u_1) \rightarrow \dots \rightarrow \varphi(u_{j+1}) \rightarrow \varphi(u'_{j+2}) \rightarrow \dots \rightarrow \varphi(u'_d)$ is pleasant.
It follows from condition~\ref{cond:omega} that $d < \omega(j+1)$.
By the definition of $\omega$ and $\OMEGAtt$ we have $T_{u_{j+1}} = T'_{u_{j+1}}$ and hence $T_{x_\ell} = T'_{x_\ell}$, and it follows that $(\varphi(a),\varphi(b)) \in E(G)$ for every edge $(a,b)$ of $T'_{x_\ell}$.

Case 2:
Assume that $\varphi(x_\ell)$ belongs to a nontrivial strongly connected component $K$ of $G$.
By conditions~\ref{cond:whirl} and \ref{cond:M}, $K$ is an $m$\hyp{}whirl for a divisor $m$ of $M$.
By condition~\ref{cond:no-SCC-SCC}, $\varphi$ maps each vertex of $T_{x_\ell}$ to $K$ or to an outlet from $K$.
Let $(a,b)$ be an edge of $T'_{x_\ell}$.
We consider the different cases according to whether $a$ and $b$ are mapped to $K$ or not.

Case 2.1:
Assume that $\varphi(a) \notin K$.
Then $h(T_a) < \OG \leq Y + 1$ by condition~\ref{cond:Y}; therefore $T_a = T'_a$ by the definition of $Y$, so $(a,b) \in E(T)$ and hence $(\varphi(a), \varphi(b)) \in E(G)$.

Case 2.2:
Assume that $\varphi(a), \varphi(b) \in K$.
Since
$d_T(a) \equiv d_{T'}(a) = d_{T'}(b) - 1 \equiv d_T(b) - 1 \pmod{M}$,
the vertices $\varphi(a)$ and $\varphi(b)$ lie in consecutive blocks of the $m$\hyp{}whirl $K$.
Therefore $(\varphi(a), \varphi(b)) \in E(G)$.

Case 2.3:
Assume that $\varphi(a) \in K$ and $\varphi(b) \notin K$.
Again by condition~\ref{cond:Y}, we have $h(T_b) < \OG \leq Y + 1$ and therefore $T_b = T'_b$.
Let $c$ be the parent of $b$ in $T$; note that $c \in V(T_{x_\ell})$.
If $c = a$, then $(\varphi(a), \varphi(b)) = (\varphi(c), \varphi(b)) \in E(G)$ and we are done.
If $c \neq a$, then $h(T_b) \geq Z \geq 0$, so there exists a path $b =: v_0 \rightarrow v_1 \rightarrow \dots \rightarrow v_Z$ in $T$.
Then $\varphi(c) \rightarrow \varphi(b) \rightarrow \varphi(v_1) \rightarrow \dots \rightarrow \varphi(v_Z)$ is a walk in $G$.
We must also have $\varphi(c) \in K$. (Suppose, to the contrary, that $\varphi(c) \notin K$. Then $h(T_c) \leq Y$ by condition~\ref{cond:Y}; hence $T_c = T'_c$ by the definition of $Y$, so $(c,b)$ is an edge of both $T$ and $T'$. This contradicts the fact that $a$ is the unique parent of $b$ in $T'$.)
Moreover,
$d_T(a) \equiv d_{T'}(a) = d_{T'}(b) - 1 \equiv d_T(b) - 1 = d_T(c) \pmod{M}$.
Therefore $\varphi(a)$ and $\varphi(c)$ belong to the same block of the $m$\hyp{}whirl $K$, and
it now follows from condition~\ref{cond:Z} that $(\varphi(a), \varphi(b)) \in E(G)$.

Case 3:
Assume that $\varphi$ maps some vertices of $T_{x_\ell}$ to nontrivial strongly connected components of $G$ but $\varphi(x_\ell)$ belongs to a trivial strongly connected component.
If $v$ is a vertex of $T_{x_\ell}$ such that $\varphi(v) \in K$, where $K$ is a nontrivial strongly connected component, and $x_1 =: u_0 \rightarrow u_1 \rightarrow \dots \rightarrow u_L \rightarrow \dots \rightarrow u_q := v$ is the path from $x_1$ to $v$ in $T$, then $\varphi(u_i) \in K$ for all $i \in \{L+1, \dots, q\}$ by conditions~\ref{cond:no-SCC-SCC} and \ref{cond:L}.
Together with condition~\ref{cond:L+1}, this implies that if $v$ and $v'$ are vertices of $T_{x_\ell}$ such that $\varphi(v) \in K$, $\varphi(v') \in K'$, where $K$ and $K'$ are nontrivial strongly connected components, then $K = K'$.
So, let us assume that $K$ is the unique nontrivial strongly connected component with nonempty intersection with $\varphi(V(T_{x_\ell}))$.
By conditions~\ref{cond:whirl} and \ref{cond:M}, $K$ is an $m$\hyp{}whirl for a divisor $m$ of $M$.
Moreover, $\varphi(u_0) \rightarrow \varphi(u_1) \rightarrow \dots \rightarrow \varphi(u_{L+1})$ is an entryway of length $L+1$, so condition~\ref{cond:L} implies $\EG = L + 1$.
Now condition~\ref{cond:shortcut} in turn implies $\lambdaG < \lambda$.

Let $x_r \in V(T'_{x_\ell})$ with $d_{T'}(x_r) = L + 1$, i.e., $x_r$ is a child of $x_\ell$ in $T'$,
and let $x_\ell =: v_0 \rightarrow v_1 \rightarrow \dots \rightarrow v_z := x_r$ be the path from $x_\ell$ to $x_r$ in $T$.
We are going to show that $(\varphi(x_\ell),\varphi(x_r)) \in E(G)$ and that $(\varphi(a),\varphi(b)) \in E(G)$ for every edge $(a,b)$ of $T'_{x_r}$.
Since $x_r$ was chosen arbitrarily among the children of $x_\ell$, this will cover all edges of $T'_{x_\ell}$ and we will be done.
We consider different possibilities.

Case 3.1:
Assume that
$\varphi(x_r) \notin K$.

Case 3.1.1:
Assume that
$\varphi(v_i) \in K$ for some $i \in \{1, \dots, z-1\}$.
Then necessarily $z > 1$; hence $d_T(x_r) > L + 1$.
In particular, $\varphi(v_1) \in K$ by condition~\ref{cond:L} and $\varphi(x_r)$ lies on an outlet, so $h(T_{x_r}) \leq Y$ by condition~\ref{cond:Y}.
Consequently, $T_{x_r} = T'_{x_r}$ by the definition of $Y$;
therefore, $(\varphi(a),\varphi(b)) \in E(G)$ for every edge $(a,b)$ of $T'_{x_r}$.
It remains to show that $(\varphi(x_\ell),\varphi(x_r)) \in E(G)$.

Observe that also $\varphi(v_{z-1}) \in K$.
(Suppose, to the contrary, that $\varphi(v_{z-1}) \notin K$.
Then a similar argument as above shows that $T_{v_{z-1}} = T'_{v_{z-1}}$.
Recall that the parent of $x_r$ in $T'$ is $x_\ell$.
Since $z > 1$, we must have $v_{z-1} \neq x_\ell$.
Consequently, $(v_{z-1}, x_r) \notin E(T'_{v_{z-1}})$, which contradicts the fact that $(v_{z-1}, x_r) \in E(T_{v_{z-1}}) = E(T'_{v_{z-1}})$.)

This means that
\begin{align*}
d_T(v_{z-1}) &= d_T(x_r) - 1 \equiv d_{T'}(x_r) - 1 = d_{T'}(x_\ell) \\ &= L = d_T(x_\ell) = d_T(v_1) - 1 \pmod{M},
\end{align*}
so $\varphi(v_{z-1})$ and $\varphi(v_1)$ lie on consecutive blocks of $K$.
Since $d_{T'}(x_r) = L + 1 < d_T(x_r)$ and $T_{x_r} = T'_{x_r}$, we have
$\lambda \leq \max(h(T_{x_r}), h(T'_{x_r})) = h(T_{x_r})$
by the definition of $\lambda$.
Therefore there exists a path $x_r \rightarrow y_1 \rightarrow \dots \rightarrow y_\lambda$ in $T$, and its image $\varphi(x_r) \rightarrow \varphi(y_1) \rightarrow \dots \rightarrow \varphi(y_\lambda)$ is a walk of length $\lambda$ in $G$.
Since $\varphi(x_1) \rightarrow \dots \rightarrow \varphi(x_\ell) \rightarrow \varphi(v_1)$ is an entryway of length $L+1 = \EG$ and we have edges $(\varphi(v_{z-1}), \varphi(v_1)), (\varphi(v_{z-1}), \varphi(x_r)) \in E(G)$,
the inequality $\lambdaG < \lambda$
implies $(\varphi(x_\ell), \varphi(x_r)) \in E(G)$, as desired.

Case 3.1.2:
Assume that $\varphi(v_i) \notin K$ for all $i \in \{1, \dots, z-1\}$.
Then actually $\varphi(x) \notin K$ for every vertex $x \in V(T_{v_1})$ (for, if there were $x \in V(T_{v_1})$ such that $\varphi(x) \in K$, then, since $d_T(v_1) = L + 1 = \EG$, we would have $\varphi(v_1) \in K$, a contradiction).
There is, however, an edge $(x_\ell, y)$ in $T$ with $\varphi(y) \in K$, so condition~\ref{cond:omega} implies that $d_T(v_1) + h(T_{v_1}) \leq \omegaG(L+1,L+1) < \omega(L+1)$ because $d_T(v_1) = L+1$.
It follows from the definition of $\omega(L+1)$ that $(d_T(v_1), h(T_{v_1})) \notin \OMEGAtt$; hence $T_{v_1} = T'_{v_1}$.
We have $x_r \in V(T_{v_1})$.
The only rooted induced subtrees of $T'_{x_\ell}$ containing the vertex $x_r$ are $T'_{x_r}$ and $T'_{x_\ell}$; hence $v_1 = x_r$ or $v_1 = x_\ell$.
The case $v_1 = x_\ell$ is impossible
because $v_1$ is the vertex following $x_\ell$ on the path from $x_\ell$ to $x_r$ in $T$;
therefore $v_1 = x_r$.
Then $(\varphi(x_\ell), \varphi(x_r)) = (\varphi(x_\ell), \varphi(v_1)) \in E(G)$.
Furthermore,
$T_{v_1} = T'_{v_1}$ implies that
$(\varphi(a), \varphi(b)) \in E(G)$ for every edge $(a,b)$ of $T'_{v_1} = T'_{x_r}$.

Case 3.2:
Assume that
$\varphi(x_r) \in K$.
Then $\varphi(v_i) \in K$ for all $i \in \{1, \dots, z\}$.
We have
\[
d_T(x_r) \equiv d_{T'}(x_r) = d_{T'}(x_\ell) + 1 = d_T(x_\ell) + 1 = d_T(v_1) \pmod{M},
\]
so $\varphi(x_r)$ and $\varphi(v_1)$ are in the same block $B_i$ of $K$.
Let $w$ be a vertex in the predecessor block $B_{i-1}$; then $w \rightarrow \varphi(x_r)$ and $w \rightarrow \varphi(v_1)$ are edges.
Since $\varphi(x_1) \rightarrow \dots \rightarrow \varphi(x_\ell) \rightarrow \varphi(v_1)$ is an entryway of length $L+1 = \EG$ and since there certainly exists a walk of length $\lambda$ starting from $\varphi(x_r)$ (just walk along vertices of $K$),
the inequality $\lambdaG < \lambda$ implies that $(\varphi(x_\ell), \varphi(x_r)) \in E(G)$.

We are going to show that $\varphi$ maps $T'_{x_r}$ homomorphically into $G$.
We go through the vertices in $T'_{x_r}$ in depth\hyp{}first\hyp{}search order, and we show that every edge of $T'_{x_r}$ is mapped to an edge of $G$.
As we will see, it suffices to go along each branch of $T'_{x_r}$ only so far until we reach a vertex $v$ such that $\varphi(v) \notin K$; once such a vertex is reached, the induced subtree rooted at $v$ will automatically be mapped homomorphically into $G$.

So, let $(a,b) \in E(T'_{x_r})$ and assume that we have already shown that every vertex on the path $x_r \rightarrow \dots \rightarrow a$ in $T'$ is mapped into $K$ by $\varphi$ and every edge along this path is mapped to an edge of $G$.
In particular, $\varphi(a) \in K$.
Let $c$ be the parent of $b$ in $T$; $(c,b) \in E(T)$.
If $a = c$, then we clearly have $(\varphi(a), \varphi(b)) = (\varphi(c), \varphi(b)) \in E(G)$.
Assume from now on that $a \neq c$.
We need to consider several cases.

Case 3.2.1:
Assume that $\varphi(b) \in K$.
Then $d_T(a) \equiv d_{T'}(a) = d_{T'}(b) - 1 \equiv d_T(b) - 1$,
that is, $\varphi(a)$ and $\varphi(b)$ lie in consecutive blocks of $K$; then clearly $(\varphi(a), \varphi(b)) \in E(G)$.

Case 3.2.2:
Assume that $\varphi(b) \notin K$.

Case 3.2.2.1:
Assume that $\varphi(c) \in K$.
Then $\varphi(b)$ lies in an outlet, so $h(T_b) \leq Y$, whence $T_b = T'_b$.
Since $a \neq c$, we have $h(T_b) \geq Z \geq 0$ by the definition of $Z$, so $G$ has an outlet of length at least $Z + 1$ starting with $\varphi(c) \rightarrow \varphi(b) \rightarrow \cdots$.
Moreover, $d_T(c) = d_T(b) - 1 \equiv d_{T'}(b) - 1 = d_{T'}(a) \equiv d_T(a) \pmod{M}$,
so $\varphi(a)$ and $\varphi(c)$ are in the same block of $K$.
Now it follows from condition~\ref{cond:Z} that $(\varphi(a), \varphi(b)) \in E(G)$.
From $T_b = T'_b$ it follows that $\varphi$ maps all edges of the subtree $T'_b$ to edges of $G$.

Case 3.2.2.2:
Assume that $\varphi(c) \notin K$.
We claim that $c = x_\ell$.
Suppose, to the contrary, that the path $x_\ell =: y_0 \rightarrow y_1 \rightarrow \dots \rightarrow y_p := c$ from $x_\ell$ to $c$ in $T$ has length $p \geq 1$.
Then $\varphi(y_i) \notin K$ for all $i \in \{0, 1, \dots, p\}$ (otherwise $\varphi(c)$ would lie in an outlet, so $h(T_c) \leq Y$, whence $T_c = T'_c$, which is clearly a contradiction since $(c,b)$ is an edge in $T$ but not in $T'$).
In fact, $\varphi(x) \notin K$ for all $x \in V(T_{y_1})$ by condition~\ref{cond:L}.
Recall the path $x_\ell =: v_0 \rightarrow v_1 \rightarrow \dots \rightarrow v_z := x_r$ in $T$.
Since $\varphi(x_\ell) \notin K$, $\varphi(x_r) \in K$, and $d_T(v_1) = L + 1$, condition \ref{cond:L} implies $\varphi(v_1) \in K$.
Then condition~\ref{cond:omega} implies that $d_T(y_1) + h(T_{y_1}) \leq \omegaG(L+1,L+1) < \omega(L+1)$; hence $(d_T(y_1), h(T_{y_1})) \notin \OMEGAtt$, so $T_{y_1} = T'_{y_1}$.
Since $(c,b)$ is an edge in $T_{y_1}$, this implies that $(c,b)$ is also an edge of $T'$, a contradiction.

Since $c = x_\ell$, we have $d_T(b) = L + 1$.
Since $\varphi(b) \notin K$, condition~\ref{cond:L} implies that $\varphi(x) \notin K$ for all $x \in V(T_b)$.
Using again the fact that $x_\ell \rightarrow v_1$ is an edge of $T$, $\varphi(v_1) \in K$, and $d_T(v_1) = L + 1$,
condition~\ref{cond:omega} implies $d_T(b) + h(T_b) \leq \omegaG(L+1,L+1) < \omega(L+1)$; hence $(d_T(b), h(T_b)) \notin \OMEGAtt$, so $T_b = T'_b$.
On the other hand, $d_{T'}(b) > L + 1$.
Therefore
$\lambda \leq \max(h(T_b), h(T'_b)) = h(T_b)$ by the definition of $\lambda$,
so there is a path of length $\lambda$ starting at $\varphi(b)$.
Furthermore,
$d_T(a) \equiv d_{T'}(a) = d_{T'}(b) - 1 \equiv d_T(b) - 1 = d_T(c) = L = d_T(v_1) - 1 \pmod{M}$,
so $\varphi(a)$ and $\varphi(v_1)$ lie in consecutive blocks of $K$, that is, $\varphi(a) \rightarrow \varphi(v_1)$ is an edge.
Now the inequality $\lambdaG < \lambda$
implies that $(\varphi(a), \varphi(b)) \in E(G)$.
From $T_b = T'_b$ it follows that $\varphi$ maps all edges of the subtree $T'_b$ to edges of $G$.

This exhausts all cases, and we conclude that $\varphi$ is a homomorphism of $T'$ to $G$.
Switching the roles of $T$ and $T'$, the same argument shows that every homomorphism of $T'$ to $G$ is a homomorphism of $T$ to $G$.
Proposition~\PROPtwoone{} now yields $\GA{G} \models t \approx t'$.
\end{proof}

%%%%%%%%%%%%%%%%%%%%%%%%%%%%%%%%%%%%%%%%%%%%%%%%%%

\section{Special cases}
\label{sec:special}

As an illustration of the parameters and results of the previous section, we now present how some of the main results of Part~I can be derived as special cases of Theorem~\ref{thm:main-parameters}.
When restricted to undirected graphs,
Theorem~\ref{thm:main-parameters} is reduced to the following proposition,
which together with Lemma~\LEMthreeone{} leads to
Theorem~\THMthreethree{}.

\begin{proposition}
Let $G$ be an undirected graph.
\begin{enumerate}[label={\upshape{(\roman*)}}]
\item If every connected component of $G$ is either trivial or a complete graph with loops, then $\GA{G}$ satisfies every bracketing identity.
\item
If every connected component is either trivial, a complete graph with loops, or a complete bipartite graph, and the last case occurs at least once, then $G$ satisfies a nontrivial bracketing identity $t \approx t'$ if and only if $\Mtt$ is even.
\item
Otherwise $G$ satisfies no nontrivial bracketing identity.
\end{enumerate}
\end{proposition}

\begin{proof}
The strongly connected components of an undirected graph are just its connected components, and every symmetric edge is part of a cycle.
Therefore, an undirected graph $G$ has no pleasant path of nonzero length and consequently no entryway nor outlet of nonzero length; thus $\PG \leq 0$, $\EG \leq 0$, $\OG \leq 0$.
It also clearly holds that $\BG = -\infty$, $\lambdaG = -\infty$, and $\omegaG(\ell,r) = -\infty$ for all $\ell, r \in \IN$ with $\ell \geq r \geq 1$.
The only whirls with symmetric edges are $1$\hyp{}whirls (i.e., complete graphs with loops) and $2$\hyp{}whirls (i.e., complete bipartite graphs).
From this it also easy to see that $\ZG = -\infty$,

For this reason, condition \ref{cond:no-SCC-SCC} of Theorem~\ref{thm:main-parameters} is automatically satisfied,
and conditions \ref{cond:H}--\ref{cond:shortcut} obviously hold for any $t, t' \in B_n$ with $t \neq t'$,
Therefore it is only conditions~\ref{cond:whirl} and \ref{cond:M} that matter.

Consider first the case that every nontrivial connected component of $G$ is a $1$\hyp{}whirl.
Then $\MG = 1$.
Since $1 \divides \Mtt$ for any $t, t' \in B_n$, $t \neq t'$, it holds that $\GA{G}$ satisfies every bracketing identity.

Consider now the case that every nontrivial connected component of $G$ is a $1$\hyp{}whirl or a $2$\hyp{}whirl and at least one of the components is a $2$\hyp{}whirl.
Then $\MG = 2$, so $\GA{G}$ satisfies a nontrivial bracketing identity $t \approx t'$ if and only if $2 \divides \Mtt$.

Finally, in the case when $G$ has a nontrivial connected component that is not a whirl, $\GA{G}$ satisfies no nontrivial bracketing identity.
\end{proof}

An equivalent characterization of associative digraphs is obtained as a special case of Theorem~\ref{thm:main-parameters}.

\begin{proposition}
\label{prop:associative-eq}
Let $G$ be a digraph.
Then $\GA{G}$ satisfies the identity $x_1 (x_2 x_3) \approx (x_1 x_2) x_3$
if and only if
the nontrivial strongly connected components of $G$ are complete graphs with loops, and for every vertex $v \in V(G)$, the outneighbourhood of $v$ is a nontrivial strongly connected component.
\end{proposition}

\begin{proof}
Denote $t := x_1 (x_2 x_3)$ and $t' := (x_1 x_2) x_3$.
It is straightforward to verify that this pair of bracketings has the following parameters (see \FIGfive):
\begin{align*}
&
\Htt = 1,
\quad
\Ltt = 0,
\quad
\Mtt = 1,
\quad
\Ytt = -1,
\quad
\Ztt = 0,
\\ &
\OMEGAtt = \{(0,2), (0,1), (1,1), (1,0)\},
\quad
\omegatt = (1,1,\dots),
\\ &
\Lambdatt = \{x_3\},
\quad
\lambdatt = 0.
\end{align*}
With these parameters, the conditions of Theorem~\ref{thm:main-parameters} for $\GA{G}$ to satify the identity $t \approx t'$ are reduced to the following:
\begin{enumerate}[label={\upshape{(\roman*)}}]
\item\label{ex:c:1} Every nontrivial strongly connected component of $G$ is a whirl.
\item\label{ex:c:2} There is no path from a nontrivial strongly connected component of $G$ to another.
\item\label{ex:c:3} $\MG = 1$.
\item\label{ex:c:4} $\PG \leq 0$.
\item\label{ex:c:5} $\EG \leq 1$.
(This follows already from \ref{ex:c:4}.)
\item\label{ex:c:6} $\OG \leq 0$.
\item\label{ex:c:7} $\ZG = -\infty$. (This is also a consequence of \ref{ex:c:1} and \ref{ex:c:6}.)
\item\label{ex:c:8} $\BG = -\infty$. In view of conditions \ref{ex:c:4} and \ref{ex:c:6}, this means that all outneighbours of a vertex belong to the same nontrivial strongly connected component.
\item\label{ex:c:9} $\omegaG(1,1) = -\infty$. (This is also a consequence of \ref{ex:c:4} and \ref{ex:c:6}.)
\item\label{ex:c:10} If $\EG = 1$, then $\lambda_G = -\infty$. This means that for any vertex $v$ belonging to a trivial strongly connected component, if $(v,u)$ is an edge, then $(v,w)$ is an edge for all vertices $w$ in the strongly connected component of $u$.
\end{enumerate}
The above conditions are easily seen to be equivalent to the following:
the nontrivial strongly connected components of $G$ are complete graphs with loops, and for every vertex $v \in V(G)$, the outneighbourhood of $v$ is an entire nontrivial strongly connected component.
\end{proof}

%%%%%%%%%%%%%%%%%%%%%%%%%%%%%%%%%%%%%%%%%%%%%%%%%%

\section{Spectrum dichotomy}
\label{sec:dichotomy}

Theorem~\ref{thm:main-parameters} provides a necessary and sufficient condition for a graph algebra to satisfy a nontrivial bracketing identity.
However, the theorem does not directly give information on the number of distinct term operations of a graph algebra induced by the bracketings of a given size.
Although a general description of the associative spectra of digraphs still eludes us,
we can find some bounds for the possible associative spectra.
In fact, as we will see in Theorem~\ref{thm dichotomy for digraphs}, the associative spectrum of a graph algebra is either constant at most $2$ or it grows exponentially.

In preparation for this dichotomy result, we shall determine the associative spectrum of the graph algebra corresponding to a certain graph on three vertices (see Proposition~\ref{prop weird spectrum}).

\begin{lemma}\label{lemma asymptotics for Rn}
For $n\geq2$ let $R_n$ be the set of words $\rho$ of length $n$ over the alphabet $\{0,1\}$ 
that satisfy the following three conditions:
\begin{enumerate}[label=\upshape{(\roman*)}]
\item\label{asym:01} $\rho$ does not start with $01$,
\item\label{asym:10} $\rho$ does not end with $10$,
\item\label{asym:101} $\rho$ does not contain $101$.
\end{enumerate}
Then $\card{R_n}$ is asymptotically $\Theta(\alpha^n)$\footnote{This means that there exist positive constants $c_1$, $c_2$ such that $c_1 \alpha^n \leq \card{R_n} \leq c_2 \alpha^n$.}, 
where $\alpha \approx 1.755$ is the unique positive root of
the polynomial $x^4 - x^3 - x^2 - 1$.
\end{lemma}

\begin{proof}
It is straightforward to verify that the map $\psi$ defined by the following formula is a bijection from $R_{n-1} \cup R_{n-2} \cup R_{n-4}$ to $R_n$ for all $n \geq 6$:
\[
\psi(\rho)  =
\begin{cases}
\rho1, & \text{if $\rho \in R_{n-1}$,} \\
\rho00, & \text{if $\rho \in R_{n-2}$,} \\
\rho1000, & \text{if $\rho \in R_{n-4}$.}
\end{cases}
\]
Thus we have the recurrence relation 
$\card{R_n} = \card{R_{n-1}} + \card{R_{n-2}} + \card{R_{n-4}}$.
The characteristic polynomial of this linear recurrence is $x^4 - x^3 - x^2 - 1$, and its roots are
\[
\alpha \approx 1.755,
\quad
\beta \approx 0.123 + 0.745i,
\quad
\gamma \approx 0.123 - 0.745i,
\quad
\delta = -1.
\]
Therefore, $\card{R_n} = a \cdot \alpha^n + b \cdot \beta^n + c \cdot \gamma^n + d \cdot \delta^n$ for suitable complex numbers $a,b,c,d$.
Since $\alpha$ is the only characteristic root of absolute value greater than one, the dominant term is $a \cdot \alpha^n$; hence we have $\card{R_n} = \Theta(\alpha^n)$. 
\end{proof}

\begin{remark}
The sequence of values $\card{R_n}$ appears as sequence \href{http://oeis.org/A005251}{A005251} in the OEIS~\cite{OEIS}.
\end{remark}

\begin{proposition}\label{prop weird spectrum}
The associative spectrum $s_n$ of the graph algebra corresponding to the graph $G$ given by $V(G) = \{u,v,w\}$, $E(G) = \{(u,v),(u,w),(w,w)\}$ is $s_n = \card{R_{n-1}}$ for all $n \geq 3$.
Hence $s_n$ is asymptotically $\Theta(\alpha^n)$.
\end{proposition}

\begin{proof}
For any DFS tree $T$ of size $n$, a map $\varphi \colon X_n \to \{u,v,w\}$ is a homomorphism of $T$ into $G$ if and only if either $\varphi(X_n) = w$, or $\varphi(x_1) = u$ and all vertices mapped to $v$ are leaves of depth one in $T$:
\[
\forall p \in X_n \colon \varphi(p) = v \implies d_{T}(p) = 1 \text{ and } h(T_p) = 0.
\]
By Proposition~\PROPtwoone, this implies that $\GA{G}$ satisfies a bracketing identity $t \approx t'$ if and only if the corresponding DFS trees $G(t)$ and $G(t')$ have the same leaves on level one. 
Thus $s_n$ counts the number of subsets of $S \subseteq \{x_2, \dots, x_n\}$ that can occur as the set of ``depth\hyp{}one leaves'' of a DFS tree of size $n$.
We claim that such sets $S$ are characterized by the following three conditions:
\begin{enumerate}[label=\upshape{(\alph*)}]
\item\label{weird:a} if $x_3 \in S$, then $x_2 \in S$;
\item\label{weird:b} if $x_{n-1} \in S$, then $x_n \in S$;
\item\label{weird:c} if $x_i, x_{i+2} \in S$, then $x_{i+1} \in S$ for all $2 \leq i \leq n - 2$.
\end{enumerate}

It is clear that these conditions are necessary. Conversely, assume that $S = \{x_{i_1}, \dots, x_{i_s}\} \subseteq \{x_2, \dots, x_n\}$ with $2 \leq i_1 < \dots < i_s \leq n$ satisfies the three conditions above.
Let us construct a DFS tree $T$ of size $n$ as follows.
For each $x_{i_k} \in S$, let $x_{i_k}$ be a child of the root $x_1$, and let $x_{i_k}$ have no children.
If $k < s$ and $i_{k+1} > i_k+1$, then let $x_{i_k+1}$ be also a child of $x_1$, and let $x_{i_k+2}, \dots, x_{i_{k+1}-1}$ be the children of $x_{i_k+1}$. 
Note that condition \ref{weird:c} guarantees that this is a nonempty set of children; hence $x_{i_k+1}$ is not a leaf.
In addition, if $x_2 \notin S$ (i.e., $i_1 > 2$), then let $x_2$ be a child of $x_1$, and let $x_3, \dots, x_{i_1-1}$ be the children of $x_2$.
Again, condition \ref{weird:a} ensures that at least $x_3$ will be a child of $x_2$, hence $x_2$ is not a leaf in this case.
Similarly, if $x_n \notin S$ (i.e., $i_s < n$), then let $x_{i_s+1}$ be a child of $x_1$, and let $x_{i_s+2}, \dots, x_n$ be the children of $x_{i_s+1}$.
Condition \ref{weird:b} guarantees that $x_{i_s+1}$ is not a leaf.
This construction yields a DFS tree $T$ whose depth\hyp{}one leaves are exactly the elements of $S$.

If we encode a set $S \subseteq \{x_2, \dots, x_n\}$ by a word $\chi \in \{0,1\}^{n-1}$ in a standard way (i.e., $\chi_i = 1$ if and only if $i + 1 \in S$), then conditions \ref{weird:a}--\ref{weird:c} translate to conditions \ref{asym:01}--\ref{asym:101} of Lemma~\ref{lemma asymptotics for Rn}. 
Thus we can conclude that $s_n = \card{R_{n-1}} = \Theta(\alpha^n)$.
\end{proof}

\begin{lemma}
\label{lem:height<2}
For $n > 1$, the number of DFS trees on $n$ vertices of height at most $2$ is $2^{n-2}$.
\end{lemma}

\begin{proof}
The depth sequence of a DFS tree on $n$ vertices of height at most $2$ is clearly an element of $\{0\} \times \{1\} \times \{1,2\}^{n-2}$, because the root $x_1$ is the only vertex at depth $0$, $x_2$ must have depth $1$, and the remaining vertices may have depth $1$ or $2$.
Conversely, every tuple $(d_1, d_2, \dots, d_n) \in \{0\} \times \{1\} \times \{1,2\}^{n-2}$
is a zag sequence and hence a depth sequence of some DFS tree by Proposition~\PROPtwosix.
The claim now follows, since DFS trees are uniquely determined by their depth sequences by Proposition~\PROPtwofive{}, and $\card{\{0\} \times \{1\} \times \{1,2\}^{n-2}} = 2^{n-2}$.
\end{proof}

\begin{lemma}\label{lemma level 1 equivalence}
Let $\sim$ be the equivalence relation on $B_n$ that relates $t$ and $t'$ if and only if $T: = G(t)$ and $T' := G(t')$ coincide up to level one, i.e.,
\[
\forall p \in X_n \colon d_{T}(p) = 1 \iff d_{T'}(p) = 1.
\]
Then $\card{B_n/{\sim}} = 2^{n-2}$ for $n \geq 2$.
\end{lemma}

\begin{proof}
We need to count sets $S\subseteq\{x_2,\dots,x_n\}$ that can occur as the set of depth\hyp{}one vertices of a DFS tree of size $n$.
Clearly, $x_2 \in S$ holds for such sets.
We claim that this condition is also sufficient.
Indeed, let $S = \{x_{i_1}, \dots, x_{i_s}\} \subseteq \{x_2, \dots, x_n\}$ with $2 = i_1 < \dots < i_s \leq n$, and let us construct a DFS tree $T$ as follows. 
For each $x_{i_k} \in S$, let $x_{i_k}$ be a child of the root $x_1$, and let $x_{i_k+1}, \dots, x_{i_{k+1}-1}$ be the children of $x_{i_k}$ (it is possible that this is an empty set of children).
Then the depth\hyp{}one vertices of $T$ are exactly the elements of $S$. 
We can conclude that $\card{B_n/{\sim}}$ is the number of subsets of $\{x_2, \dots, x_n\}$ that contain $x_2$, and this is obviously $2^{n-2}$.
\end{proof}

By a \emph{directed bipartite graph} we mean a bipartite graph $G = (V,E)$ with bipartition $V = V_1 \cup V_2$ such that $E \subseteq V_1 \times V_2$ (i.e., all edges go to the ``same direction'').
The \emph{weakly connected components} of a digraph $G$ are its induced subgraphs on (the vertex sets of) the connected components of the underlying undirected graph of $G$.

\begin{theorem}\label{thm dichotomy for digraphs}
For any digraph $G$ we have the following three mutually exclusive cases.
\begin{enumerate}[label=\upshape{(\roman*)}]
\item\label{dichotomy:1}
The associative spectrum of $\GA{G}$ is constant $1$. 
These digraphs are characterized in Proposition~\PROPfourone{} or, equivalently, in Proposition~\ref{prop:associative-eq}.
\item\label{dichotomy:2}
The associative spectrum of $\GA{G}$ is constant $2$. 
This holds if and only if each weakly connected component of $G$ is either associative or a directed bipartite graph with at least one edge, and the latter occurs at least once.
\item\label{dichotomy:weird}
In all other cases the associative spectrum of $\GA{G}$ is bounded below by the spectrum of the graph given in Proposition~\ref{prop weird spectrum}, i.e.,  $s_n(\GA{G}) \geq \card{R_{n-1}} = \Theta(\alpha^n)$ \textup{(}cf.\ Lemma~\ref{lemma asymptotics for Rn}\textup{)}.
\end{enumerate}
\end{theorem}

\begin{proof}
Let $G$ be an arbitrary digraph, and let $s_n = s_n(\GA{G})$ denote the associative spectrum and $\sigma_n = \sigma_n(\GA{G})$ denote the fine associative spectrum of the corresponding graph algebra. 
Let us assume that $s_n$ does not grow exponentially.
Then $G$ satisfies conditions \ref{cond:whirl} and \ref{cond:no-SCC-SCC} of Theorem~\ref{thm:main-parameters} (otherwise the associative spectrum would consist of the Catalan numbers). 
If $\MG \geq 2$, then $G$ contains an induced subgraph that is isomorphic to the directed cycle $\GRcycle{m}$ for some $m \geq 2$; hence $s_n \geq s_n(\GRcycle{m}) \geq s_n(\GRcycle{2}) = 2^{n-2}$ by Proposition~\PROPfivefour{} and Remark~\REMfivefive, contradicting our assumption on the growth of the spectrum.
If $\PG \geq 2$, then condition \ref{cond:H} of Theorem~\ref{thm:main-parameters} shows that all bracketings corresponding to DFS trees of height at most $2$ fall into different equivalence classes of the fine spectrum $\sigma_n$.
Therefore, Lemma~\ref{lem:height<2} implies that $s_n \geq 2^{n-2}$, a contradiction. 
If $\EG \geq 2$, then by condition \ref{cond:L} of Theorem~\ref{thm:main-parameters}, bracketings $t, t' \in B_n$ fall into different equivalence classes of the fine spectrum whenever the corresponding DFS trees differ at level one.
Hence, by Lemma~\ref{lemma level 1 equivalence}, we have $s_n \geq 2^{n-2}$, which is a contradiction again.
A similar argument using condition \ref{cond:Y} of Theorem~\ref{thm:main-parameters} and Lemma~\LEMfivesix{} shows that $\OG \geq 1$ also leads to the contradiction $s_n \geq 2^{n-2}$.

We have proved thus far that if $\GA{G}$ has a subexponential spectrum, then $G$ satisfies conditions \ref{cond:whirl} and \ref{cond:no-SCC-SCC} of Theorem~\ref{thm:main-parameters} and the (in)equalities $\MG = 1$, $\PG \leq 1$, $\EG \leq 1$, $\OG \leq 0$.
Let us assume that the latter hold, and let $V_0$ be the union of the vertex sets of the nontrivial strongly connected components of $G$ (if there are any).
From $\PG \leq 1$, $\EG \leq 1$ and $\OG \leq 0$ we can see that no vertex of $V \setminus V_0$ can have an inneighbour and an outneighbour at the same time.
Let $V_1$ be the set of vertices that have an outneighbour, and let $V_2 := V \setminus (V_0 \cup V_1)$.
Thus $V = V_0 \cup V_1 \cup V_2$ (some of these sets might be empty), and the subgraph induced on $V_1 \cup V_2$ is a directed bipartite graph, whereas the subgraph induced on $V_0$ is a disjoint union of complete graphs with loops by conditions \ref{cond:whirl} and \ref{cond:no-SCC-SCC} of Theorem~\ref{thm:main-parameters} and by $\MG = 1$.
Since $\OG \leq 0$, there is no edge from $V_0$ to $V_1 \cup V_2$, and there is no edge from $V_2$ to $V_0$ by the definition of $V_2$, but we may have edges from $V_1$ to $V_0$.

Let $(v_1, v_0)$ be such an edge (i.e., $v_1 \in V_1$ and $v_0 \in V_0$).
If $v'_0$ is another vertex in the strongly connected component of $v_0$, then we must have the edge $(v_1, v'_0)$.
Indeed, if this was not the case, then subgraph induced on $\{v_1, v_0, v'_0\}$ would be isomorphic to the graph of Proposition~\PROPfivenine, and it has an exponential spectrum.
(Note that the spectrum of any induced subgraph provides a lower estimate of the spectrum of the whole graph.)
On the other hand, if $v'_0$ belongs to another nontrivial strongly connected component, then the presence of the edge $(v_1, v'_0)$ would give rise to an induced subgraph isomorphic to that of Proposition~\PROPfiveeight, again contradicting our assumption about the subexponential growth of the spectrum.
Thus we have proved that if a vertex of $V_1$ has outneighbours in $V_0$, then these outneighbours form a nontrivial strongly connected component.

Finally, if a vertex $v_1 \in V_1$ has an outneighbour $v_0 \in V_0$ and also an outneighbour $v_2 \in V_2$, then the subgraph induced on $\{v_1, v_2, v_0\}$ is isomorphic to the graph of Proposition~\ref{prop weird spectrum}, forcing again an exponential spectrum.
Thus some vertices of $V_1$ have outneighbours only in $V_0$, while others have outneighbours only in $V_2$.
The former vertices together with $V_0$ form an associative graph (see Proposition~\ref{prop:associative-eq}), while the latter vertices together with $V_2$ form a directed bipartite graph.
This proves that every digraph with a subexponential associative spectrum belongs to cases \ref{dichotomy:1} or \ref{dichotomy:2} of the current theorem.

It only remains to prove that the spectrum of a directed bipartite graph with at least one edge is constant $2$.
But this is easily done with the help of Theorem~\ref{thm:main-parameters}.
All conditions except for \ref{cond:H} are satisfied trivially for all $t, t' \in B_n$ with $t \neq t'$.
Condition \ref{cond:H} gives $1 = \PG < \Htt$, which means that $\sigma_n$ has two equivalence classes: $\{t\}$ and $B_n \setminus \{t\}$, where 
$t = ((\cdots((x_1 x_2) x_3) \cdots) x_{n-1}) x_n$ is the bracketing that corresponds to the unique DFS tree of size $n$ and height $1$.
\end{proof}

\begin{remark}
Theorem~\ref{thm dichotomy for digraphs} implies that there are only two different bounded spectra of graph algebras, namely constant $1$ and constant $2$.
For arbitrary groupoids, all sequences of the form $(2, \dots, 2, 1, 1, \dots)$ can occur as associative spectra, and there are other bounded spectra (e.g., constant $3$), too \cite{CsaWal-2000}.
Theorem~\ref{thm dichotomy for digraphs} also implies that unboundeded spectra of graph algebras grow exponentially, the smallest growth rate being $\Theta(\alpha^n)$.
This is not true for arbitrary groupoids either: there exist groupoids with polynomial spectra of arbitrary degrees \cite{LieWal-2009}.
\end{remark}

%%%%%%%%%%%%%%%%%%%%%%%%%%%%%%%%%%%%%%%%%%%%%%%%%%

\section*{Acknowledgments}

The authors would like to thank Mikl\'{o}s Mar\'{o}ti and Nikolaas Verhulst for helpful discussions.

This work is funded by National Funds through the FCT -- Funda\c{c}\~ao para a Ci\^encia e a Tecnologia, I.P., under the scope of the project UIDB/00297/2020 (Center for Mathematics and Applications) and the project PTDC/MAT-PUR/31174/\discretionary{}{}{}2017.
Research partially supported by the Hungarian Research, Development and Innovation Office grant K115518, and by grant TUDFO/47138-1/2019-ITM of the Ministry for Innovation and Technology, Hungary.

%%%%%%%%%%%%%%%%%%%%%%%%%%%%%%%%%%%%%%%%%%%%%%%%%%

\end{document}